\documentclass[10pt]{article}
\pdfoutput=1
\usepackage{mathrsfs}
\usepackage{amsfonts}
\usepackage{amsmath,amssymb,mathrsfs,bm,latexsym,graphicx,color}
\newtheorem{The}{Theorem}[section]
\newtheorem{Pro}{Proposition}[section]
\newtheorem{Lem}{Lemma}[section]
\newtheorem{Cor}{Corollary}[section]
\newtheorem{Def}{Definition}[section]
\newtheorem{Exa}{Example}[section]
\newtheorem{Rem}{Remark}[section]

\numberwithin{equation}{section}
\numberwithin{The}{section}
\allowdisplaybreaks 

\newenvironment {proof} {\noindent {\bf Proof.}}{\quad $\square$\par\vspace{3mm}}

\begin{document}

\title{ On some properties of three different types of triangular blocked tensors
\thanks{This work was supported by the NSF of China (Grant Nos. 11231004, 11271288 and 11571123)
and the NSF of Guangdong Provincial (Grant No. 2015A030313377).}}
\author{Jiayu Shao$^{a,}$\footnote{{\it{Email address:\;}}jyshao@tongji.edu.cn}
\qquad Lihua You$^{b,}$\footnote{{\it{Corresponding author:\;}}ylhua@scnu.edu.cn}}
\vskip.2cm
\date{{\small
$^{a}$ Department of Mathematics, Tongji University, Shanghai, 200092, P.R. China\\
$^{b}$ School of Mathematical Sciences, South China Normal University, Guangzhou, \\
510631, P.R. China\\
}}
\maketitle


\begin{abstract}

We define three types of upper (and lower) triangular blocked tensors, which are all generalizations of the triangular blocked matrices. We study some basic properties and characterizations of these three types of triangular blocked tensors. We obtain the formulas for the determinants, characteristic polynomials and spectra of the first and second type triangular blocked tensors, and give an example to show that these formulas no longer hold for the third type triangular blocked tensors. We prove that the product of any two $(n_1,\cdots,n_r)$-upper (or lower) triangular blocked tensors of the first or second or third type is still an $(n_1,\cdots,n_r)$-upper (or lower) triangular blocked tensor of the same type. We also prove that, if an $(n_1,\cdots,n_r)$-upper triangular blocked tensor of the first or second or third type has a left $k$-inverse, then its unique left $k$-inverse is still an $(n_1,\cdots,n_r)$-upper triangular blocked tensor of the same type. Also if it has a right $k$-inverse, then all of its right $k$-inverses are still $(n_1,\cdots,n_r)$-upper triangular blocked tensors of the same type. By showing that the left $k$-inverse (if any) of a weakly irreducible nonsingular $M$-tensor is a positive tensor, we show that the left $k$-inverse (if any) of a first or second or third type canonical $(n_1,\cdots,n_r)$-upper triangular blocked nonsingular $M$-tensor is an $(n_1,\cdots,n_r)$-upper triangular blocked tensor of the same type all of whose diagonal blocks are positive tensors.  We also show that every order $m$ dimension $n$ tensor is permutation similar to some third type normal upper triangular blocked tensor (all of whose diagonal blocks are irreducible). We give an example to show that this is not true for the first type canonical upper triangular blocked tensor.



\end{abstract}

\section{Introduction}
\hskip0.6cm In recent years, the study of tensors and the spectra of tensors (and hypergraphs) with their various applications has attracted
extensive attention and interest, since the work of L.Qi (\cite{2005Q}) and L.H.Lim (\cite{2005L}) in 2005.

As was in \cite{2005Q}, an order $m$ dimension $n$ tensor $\mathbb {A}=(a_{i_1i_2\cdots i_m})_{1\le i_j\le n \ (j=1,\cdots ,m)}$ over the
complex field $\mathbb {C}$ is a multidimensional array with all entries
$a_{i_1i_2\cdots i_m}\in \mathbb {C} \  \ (i_1,\cdots ,i_m\in [n]=\{1,\cdots ,n\})$.

In this paper, we define and study three different types of the general $(n_1,\cdots,n_r)$-upper (and lower) triangular blocked tensors.

It is well known that the triangular blocked matrices are very important and useful in the study and applications of matrices.
For tensors, Hu et al \cite{2013HHLQ} gave a determinant formula for the special case $r=2$ of some type of the $(n_1,\cdots,n_r)$-triangular blocked tensors. Also, Shao et al \cite{2013SSZ} defined a type of lower triangular blocked tensors which is essentially equivalent to the second type upper triangular blocked tensors defined in this paper (see Theorem \ref{The25} of this paper for the proof). Shao et al \cite{2013SSZ} also studied some other basic properties of such type triangular blocked tensors.

Recently, Hu, Huang and Qi in \cite{2014HHQ} introduced and studied the ``nonnegative tensor partition" which is also essentially equivalent to the second type  triangular blocked tensors defined in this paper (up to a permutation similarity). They obtained that every order $m$ dimension $n$ tensor is permutation similar to such a type of upper triangular blocked tensor each of whose diagonal blocks are weakly irreducible (also see Proposition 1 of \cite{2016HQ}). Hu and Qi \cite{2016HQ} also used this ``nonnegative tensor partition" to study some spectral properties of nonnegative tensors, and obtain a necessary and sufficient condition for a nonnegative tensor to have a positive eigenvector.

In this paper, we first give the definitions of three different types of the general $(n_1,\cdots,n_r)$-upper (and lower)
triangular blocked tensors, which are all the natural generalizations of the $(n_1,\cdots,n_r)$-upper (and lower) triangular
blocked matrices. Then we study some properties of these three types of triangular blocked tensors.

In some sense, the first and second type upper triangular blocked tensors defined in this paper are generalizations of the weakly reducible tensors (see Definition 1.1 below) which corresponds to the case $r=2$ (with two diagonal blocks) in Definitions 2.1 and 2.2 up to a permutation similarity (also see Remarks 2.1 and 2.2 in \S 2), while the third type upper triangular blocked tensors defined in this paper is a generalization of the reducible tensors (see Definition 1.1 below) which corresponds to the case $r=2$ in Definitions 2.3 up to a permutation similarity (also see Remark 2.4 in \S 2).

\vskip 0.1cm

We first study some basic properties and characterizations of these three types of triangular blocked tensors in \S 2.
Then in \S 3, we obtain the formulas for the determinants, characteristic polynomials and spectra of the first two types of
triangular blocked tensors, and give an example to show that these formulas no longer hold for the third type triangular blocked tensors.
We prove in \S 4 that the product of any two $(n_1,\cdots,n_r)$ upper (or lower) triangular blocked tensors of the first or second or third type
is still an $(n_1,\cdots,n_r)$ upper (or lower) triangular blocked tensor of the same type.

We also prove in \S 5 that, if an $(n_1,\cdots,n_r)$ upper (or lower) triangular blocked tensor of the first or second or third type has a
left $k$-inverse, then its unique left $k$-inverse is still an $(n_1,\cdots,n_r)$ upper (or lower) triangular blocked tensor of all the three types.
Also if it has a right $k$-inverse, then all of its right $k$-inverses are still
$(n_1,\cdots,n_r)$ upper (or lower) triangular blocked tensors of all the three types.
Furthermore, by showing (in Theorem \ref{The53}) that the left $k$-inverse (if any) of a weakly irreducible nonsingular $M$-tensor is a positive tensor, we show that the left $k$-inverse (if any) of a first or second or third type canonical $(n_1,\cdots,n_r)$-upper (or lower) triangular
blocked nonsingular $M$-tensor is an $(n_1,\cdots,n_r)$-upper (or lower) triangular blocked tensor of all the three types all of
whose diagonal blocks are positive tensors.

\vskip 0.08cm

In \S 6, we show that every order $m$ dimension $n$ tensor is permutation similar to some third type normal upper triangular blocked tensor
(all of whose diagonal blocks are irreducible). On the other hand, we also give an example to show that not every tensor can be permutational
similar to some first type normal upper triangular blocked tensor (all of whose diagonal blocks are weakly irreducible).

\vskip 0.1cm

Now we introduce some basic concepts of tensors which are relevant to the main results of this paper.

\noindent
\begin{Def}\label{Def11}{\rm(\cite{2008CPZ, 2013FGH, 2011YY})} 
Let $\mathbb {A}$ be an order $m$ dimension $n$ tensor.

\noindent {\rm(1)} If there exists a proper subset $I$ of the set $[n]$ such that
\begin{equation}\label{eq11}
a_{i_1i_2\cdots i_m}=0 \quad (\forall \ i_1\in I, \  \mbox {and all of the}  \ i_2,\cdots, i_m\notin I),  
\end{equation}

\noindent then $\mathbb {A}$ is called reducible (or sometimes $I$-reducible). If $\mathbb {A}$ is not reducible,
then $\mathbb{A}$ is called irreducible.
\vskip 0.1cm

\noindent {\rm(2)}  If there exists a proper subset $I$ of the set $[n]$ such that
\begin{equation}\label{eq12}
a_{i_1i_2\cdots i_m}=0 \quad (\forall \ i_1\in I, \  \mbox {and at least one of the}  \ i_2,\cdots, i_m\notin I), 
\end{equation}

\noindent then $\mathbb {A}$ is called weakly reducible (or sometimes $I$-weakly reducible).
If $\mathbb {A}$ is not weakly reducible, then $\mathbb{A}$ is called weakly irreducible.
\end{Def}

The following is an equivalent definition of the determinant of an order $m$ dimension $n$ tensor with $m\ge 2$
(see \cite{2013HHLQ} and \cite{2013SSZ}).

\vskip 0.1cm

\begin{Def}\label{Def12}{\rm(\cite{2013HHLQ, 2013SSZ})} Let $\mathbb {A}$ be an order $m$ dimension $n$ tensor with $m\ge 2$.
Then $\det(\mathbb {A})$ is the unique polynomial on the entries of $\mathbb {A}$ satisfying the following three conditions:

\vskip 0.01cm

\noindent {\rm(1)} $\det(\mathbb {A})=0$ if and only if the system of homogeneous equations $\mathbb {A}x=0$ has a nonzero solution.

\vskip 0.01cm

\noindent {\rm(2)} $\det(\mathbb {I})=1$, where $\mathbb {I}$ is the unit tensor.

\vskip 0.01cm

\noindent {\rm(3)} $\det(\mathbb {A})$ is an irreducible polynomial on the entries of $\mathbb {A}$,
when the entries $a_{i_1\cdots i_m} \ (1\le i_1,\cdots, i_m\le n)$ of $\mathbb {A}$ are all viewed as independent distinct variables.
\end{Def}

By using the definition of determinants, we can define the characteristic polynomial $\phi_{\mathbb {A}}(\lambda)$ of a tensor $\mathbb {A}$
to be the determinant $\det(\lambda \mathbb {I}- \mathbb {A})$, where $\mathbb {I}$ is the unit tensor. The spectrum of the tensor
$\mathbb {A}$ is defined to be the multi-set of the roots of the characteristic polynomial of $\mathbb {A}$.

The following definitions about the product of tensors and the permutation similarity of tensors can be found in \cite{2013S}.

\begin{Def}\label{Def13}{\rm(\cite{2013S})}
Let $\mathbb {A}$ (and $\mathbb {B}$) be an order $m\ge 2$ (and order $k\ge 1$), dimension $n$ tensor, respectively.
Define the product $\mathbb {A}\mathbb {B}$ to be the following tensor $\mathbb {C}$ of order $(m-1)(k-1)+1$ and dimension $n$:
\begin{equation}\label{eq13}
c_{i\alpha_1\cdots \alpha_{m-1}}=\sum_{i_2,\cdots ,i_m=1}^n a_{ii_2\cdots i_m}b_{i_2\alpha_1}\cdots b_{i_m\alpha_{m-1}}
\quad  (i\in [n], \  \alpha_1, \cdots ,\alpha_{m-1}\in [n]^{k-1}). 
\end{equation}
\end{Def}

\begin{Def}\label{Def13}{\rm(\cite{2013S})} Let $\mathbb{A}$ and $\mathbb {B}$ be two order $m$ dimension $n$ tensors.
If there exists a permutation matrix $P$ such that $\mathbb {B}=P\mathbb {A}P^T$,
then we say that $\mathbb {A}$ and $\mathbb {B}$ are permutational similar. \end{Def}

It was proved in \cite{2013S}  that similar (and thus permutational similar) tensors have the same determinants,
same characteristic polynomials and the same spectra.

Let $I$ be a nonempty subset of the set $[n]$. Then the (principal) subtensor $\mathbb{A}[I]$ of $\mathbb{A}$ is the
subtensor consisting of those entries $a_{i_1i_2\cdots i_m}$ of $\mathbb{A}$ all of whose subscripts $i_1,i_2,\cdots, i_m$ are in $I$.

\section{The definitions of the three types of triangular blocked tensors and their basic properties}
\hskip.6cm We first recall the definition of the triangular blocked matrices.

Let $A$ be a matrix of order $n$. Let $n_1,\cdots,n_r \ (r\ge 2)$ be positive integers with $n_1+\cdots+n_r=n$. Write $S_0=0$ and
$$S_j=n_1+\cdots+n_j, \qquad  I_j=\{S_{j-1}+1,\cdots,S_j\} \quad (j=1,\cdots,r).$$
If for any $j=2,\cdots,r$, we have
$$a_{ii_2}=0 \quad (\forall i\in I_j, \ \mbox {and} \ i_2\le S_{j-1}), $$
then we say that $A$ is an $(n_1,\cdots,n_r)$-upper triangular blocked matrix.

$$\left(\begin{array}{ccccc}
A_{1}&\cdots &\ast &\cdots &\ast \\
\vdots&\ddots &\vdots &\cdots &\ast \\
O&\cdots &A_{j}&\cdots &\ast \\
\vdots&\cdots &\vdots &\ddots &\vdots \\
O&\cdots &O &\cdots &A_r \\
\end{array} \right)
$$
In the following, we want to generalize this definition from matrices to tensors.

\subsection{The definitions of the three types of triangular blocked tensors}

\hskip.6cm In this subsection, we give the definitions of the three types of triangular blocked tensors. Here the definition of our second type upper triangular blocked tensors can be proved to be equivalent to the Definition 4.3 of \cite{2013SSZ} on a type of lower triangular blocked tensors  (see Theorem \ref{The25} of this paper), and is also essentially equivalent to the ``nonnegative tensor partition" defined and studied by Hu, Huang and Qi in \cite{2014HHQ}, and by Hu and Qi in \cite{2016HQ} (up to a permutation similarity).

\begin{Def}\label{Def21}
Let $\mathbb {A}$ be an order $m$ and dimension $n$ tensor, $n_1,\cdots,n_r \ (r\ge 2)$ be positive integers with $n_1+\cdots+n_r=n$.
Write $S_0=0$ and
$$S_j=n_1+\cdots+n_j, \qquad  I_j=\{S_{j-1}+1,\cdots,S_j\} \quad (j=1,\cdots,r).$$

\noindent {\rm(1)} If for any $j=2,\cdots,r$, we have
\begin{equation}\label{eq21}
a_{ii_2\cdots i_m}=0 \quad (\forall i\in I_j, \ \mbox {and} \ \min\{i_2,\cdots,i_m\}\le S_{j-1}), 
\end{equation}
then we say that $\mathbb {A}$ is a  first type $(n_1,\cdots,n_r)$-upper triangular blocked tensor, or $(n_1,\cdots,n_r)$-1stUTB tensor (or simply $(n_1,\cdots,n_r)$-UTB tensor) with diagonal
blocks $\mathbb {A}_1,\cdots,\mathbb {A}_r$, where $\mathbb {A}_j=\mathbb {A}[I_j]$ is the $j$-th diagonal block of this UTB tensor.

\noindent  {\rm (Here the condition $\min\{i_2,\cdots,i_m\}\le S_{j-1}$ means that at least one index of $i_2,\cdots,i_m$ belongs to $I_1\cup\cdots\cup I_{j-1}$.)}

\vskip 0.1cm

\noindent {\rm(2)} If for any $j=1,\cdots,r-1$, we have
\begin{equation}\label{eq22}
a_{ii_2\cdots i_m}=0 \quad (\forall i\in I_j, \ \mbox {and} \ \max\{i_2,\cdots,i_m\}\ge S_{j}+1), 
\end{equation}
then we say that $\mathbb {A}$ is a first type $(n_1,\cdots,n_r)$-lower triangular blocked tensor (or simply $(n_1,\cdots,n_r)$-LTB tensor) with diagonal blocks $\mathbb {A}_1,\cdots,\mathbb {A}_r$.

\noindent {\rm(3)} If for any $j=1,\cdots,r$, we have
\begin{equation}\label{eq23}
a_{ii_2\cdots i_m}=0 \quad (\forall i\in I_j, \ \mbox {and at least one of } \ i_2,\cdots,i_m \notin I_j), 
\end{equation}
then we say that $\mathbb {A}$ is an $(n_1,\cdots,n_r)$-diagonal blocked tensor with diagonal blocks $\mathbb {A}_1,\cdots,\mathbb {A}_r$.
\end{Def}

There of course might be other possible ways to define the triangular blocked tensors. For example,
the following (second type) definition is equivalent to Definition 4.3 in \cite{2013SSZ}  (see Theorem \ref{The25} below), and is also essentially equivalent to the ``nonnegative tensor partition" defined and studied by Hu, Huang and Qi in \cite{2014HHQ}, and by Hu and Qi in \cite{2016HQ} (up to a permutation similarity).

\begin{Def}\label{Def22}  Let $\mathbb {A}$ be an order $m$ and dimension $n$ tensor,
$n_1,\cdots,n_r \ (r\ge 2)$ be positive integers with $n_1+\cdots+n_r=n$. Write $S_0=0$ and
$$S_j=n_1+\cdots+n_j, \qquad  I_j=\{S_{j-1}+1,\cdots,S_j\} \quad (j=1,\cdots,r).$$

\noindent  {\rm(1)} If for any $j=2,\cdots,r$, we have
\begin{equation}\label{eq24}
a_{ii_2\cdots i_m}=0 \quad (\forall i\in I_j, \ \min\{i_2,\cdots,i_m\}\le S_{j-1} \  \mbox {and} \ \max\{i_2,\cdots,i_m\}\le S_{j}), 
\end{equation}
then $\mathbb {A}$ is called a second type $(n_1,\cdots,n_r)$-upper triangular blocked tensor (or 2ndUTB tensor).

\noindent {\rm(Here the condition $\min\{i_2,\cdots,i_m\}\le S_{j-1}$ means that at least one index of $i_2,\cdots,i_m$ belongs to $I_1\cup\cdots\cup I_{j-1}$, while the condition $\max\{i_2,\cdots,i_m\}\le S_{j}$ means that all indices of $i_2,\cdots,i_m$ belong to $I_1\cup\cdots\cup I_{j}$.)}

\noindent {\rm(2)} If for any $j=1,\cdots,r-1$, we have
\begin{equation}\label{eq25}
a_{ii_2\cdots i_m}=0 \quad (\forall i\in I_j, \ \max\{i_2,\cdots,i_m\}\ge S_{j}+1 \  \mbox {and} \ \min\{i_2,\cdots,i_m\}\ge S_{j-1}+1), 
\end{equation}
then $\mathbb {A}$ is called a second type $(n_1,\cdots,n_r)$-lower triangular blocked tensor (or 2ndLTB tensor).
\end{Def}

\begin{Rem}\label{Rem21}
 From Definitions \ref{Def21} and \ref{Def22} we can see that, in the case $r=2$, and thus $j=2$ is the only case for $2\le j\le r$,
 we see that the condition $\max\{i_2,\cdots,i_m\}\le S_j=S_r=n$ holds automatically. So in this case $r=2$,
 the $(n_1,n_2)$-UTB tensors and $(n_1,n_2)$-2ndUTB tensors are the same.
\end{Rem}

\begin{Rem}\label{Rem22}
Also it is easy to see that, in the case $r=2$, if $\mathbb {A}$ is an $(n_1,n_2)$-UTB (and 2ndUTB) tensor,
then $\mathbb {A}$ is $I_2$-weakly reducible. On the other hand,
a tensor $\mathbb{A}$ of dimension $n$ is weakly reducible if and only if there exists some integer $k$ with $1\le k\le n-1$
such that $\mathbb{A}$ is permutation similar to some $(k,n-k)$-UTB (and 2ndUTB) tensor.
\end{Rem}

\begin{Exa}\label{Exa21}
 For the special case $r=2$, $n_1=k$ and $n_2=n-k \ (1\le k\le n-1)$, we have that $\mathbb {A}$ is a $(k,n-k)$-UTB (and 2ndUTB) tensor
 if and only if the following condition holds:
\begin{equation}\label{eq26} a_{ii_2\cdots i_m}=0 \quad (\forall i>k, \ \mbox {and} \ \min\{i_2,\cdots,i_m\}\le k).
\end{equation}
\end{Exa}
\begin{proof}  Since $r=2$, we only need to check that (\ref{eq21}) holds for the case $j=2$. Now in this case we have
$i\in I_2\Longleftrightarrow i>k$, and $\min\{i_2,\cdots,i_m\}<S_{j-1}+1=k+1\Longleftrightarrow \min\{i_2,\cdots,i_m\}\le k$.
So in this case (\ref{eq21}) is equivalent to (\ref{eq26}).
\end{proof}

Now we define the third type triangular blocked tensors.

\begin{Def}\label{Def23} Let $\mathbb {A}$ be an order $m$ and dimension $n$ tensor,
$n_1,\cdots,n_r \ (r\ge 2)$ be positive integers with $n_1+\cdots+n_r=n$. Write $S_0=0$ and
$$S_j=n_1+\cdots+n_j, \qquad  I_j=\{S_{j-1}+1,\cdots,S_j\} \quad (j=1,\cdots,r).$$

\noindent  {\rm(1)} If for any $j=2,\cdots,r$, we have
\begin{equation}\label{eq27} a_{ii_2\cdots i_m}=0 \quad (\forall i\in I_j, \ \mbox {and} \ \max\{i_2,\cdots,i_m\}\le S_{j-1}), \end{equation} 
then we say that $\mathbb {A}$ is a third type $(n_1,\cdots,n_r)$-upper triangular blocked tensor (or 3rdUTB tensor).

\noindent  {\rm(Here the condition $\max\{i_2,\cdots,i_m\}\le S_{j-1}$ means that all indices of $i_2,\cdots,i_m$ belong to $I_1\cup\cdots\cup I_{j-1}$.)}

\vskip 0.1cm

\noindent  {\rm(2)}  If for any $j=1,\cdots,r-1$, we have
\begin{equation}\label{eq28} a_{ii_2\cdots i_m}=0 \quad (\forall i\in I_j, \ \mbox {and} \ \min\{i_2,\cdots,i_m\}\ge S_{j}+1), \end{equation} 
then we say that $\mathbb {A}$ is a third type $(n_1,\cdots,n_r)$-lower triangular blocked tensor (or 3rdLTB tensor).
\end{Def}

\begin{Rem}\label{Rem23}  It is easy to see from Definition 2.3 that, in the case $r=2$,
if $\mathbb {A}$ is an $(n_1,n_2)$-3rdUTB tensor, then $\mathbb {A}$ is $I_2$-reducible. On the other hand,
a tensor $\mathbb{A}$ of dimension $n$ is reducible if and only if there exists some integer $k$ with $1\le k\le n-1$ such that $\mathbb{A}$
is permutation similar to some $(k,n-k)$-3rdUTB tensor.
\end{Rem}

Now we give some more remarks on the definitions of the above three types of upper triangular blocked tensors.

\begin{Rem}\label{Rem24}

\noindent  {\rm(1)}  It is easy to see from the definitions that an $(n_1,\cdots,n_r)$-UTB tensor is always an $(n_1,\cdots,n_r)$- 2ndUTB tensor,
and an $(n_1,\cdots,n_r)$-2ndUTB tensor is always an $(n_1,\cdots,n_r)$- 3rdUTB tensor.

Also it is easy to see that, in this case $m=2$, all the three types of $(n_1,\cdots,n_r)$-UTB tensors are the same.

\noindent  {\rm(2)} In case of $m=2$, namely if the tensor $\mathbb{A}$ is a matrix of order $n$,
then our definitions of all the three types are equivalent to the definition of upper (or lower) triangular blocked matrices.
So these definitions are all generalizations of triangular blocked matrices.

\noindent  {\rm(3)} If $n_1=\cdots=n_r=1$, then an $(n_1,\cdots,n_r)$-UTB (or LTB) tensor is called an upper (or lower) triangular tensor.
We also have the similar remarks for the second and third types $(n_1,\cdots,n_r)$-UTB (or LTB) tensors.

\noindent  {\rm(4)} Up to some permutation similarity, Definition \ref{Def22}  is essentially equivalent to the (nonnegative) tensor partition
defined in \cite{2016HQ} (Proposition 1). Also, \cite{2016HQ} mentioned that  Proposition 1 in \cite{2016HQ} is the main result of \cite{2014HHQ} .

\noindent  {\rm(5)} It is easy to see from the definitions that, an $(n_1,\cdots,n_r)$-UTB (or 2ndUTB) tensor is permutation similar to
an $(n_r,\cdots,n_1)$-LTB (or 2nd LTB) tensor via the permutation $\sigma : [n]\rightarrow [n]$ with $\sigma(i)=n+1-i$.

\noindent  {\rm(6)} It is easy to see from the definitions that $\mathbb {A}$ being a first or second or third type $(n_1,\cdots,n_r)$-UTB tensor
is independent of the diagonal blocks $\mathbb {A}_1,\cdots,\mathbb {A}_r$.
\end{Rem}

The following proposition is an easy consequence of Definition \ref{Def21}.

\begin{Pro}\label{Pro21}
$\mathbb {A}$ is an $(n_1,\cdots,n_r)$-diagonal blocked tensor if and only if $\mathbb {A}$ is both an $(n_1,\cdots,n_r)$-UTB
tensor and an $(n_1,\cdots,n_r)$-LTB tensor.\end{Pro}

The following example shows that Proposition \ref{Pro21} would no longer hold for the second and third type triangular blocked tensors.

\begin{Exa}\label{Exa22} Let $\mathbb {A}$ be an $(n_1,\cdots,n_r)$-diagonal blocked tensor of dimension $n$ and order 3. Take
$$2\le j\le r-1, \quad i\in I_j, \quad p\le S_{j-1} \quad \mbox {and} \quad q>S_j.$$

Let $\mathbb {B}$ be the tensor of dimension $n$ and order 3 with $b_{ipq}=1$ and all the other entries 0.
Then the tensor $\mathbb {A}+\mathbb {B}$ is both an $(n_1,\cdots,n_r)$-2ndUTB (and thus 3rdUTB) tensor and an
$(n_1,\cdots,n_r)$-2ndLTB (and thus 3rdLTB) tensor, but it is not an $(n_1,\cdots,n_r)$-diagonal blocked tensor.
\end{Exa}

\begin{Exa}\label{Exa23} If $G$ is a $k$-uniform hypergraph with $r$ connected components $G_1,\cdots,G_r$.
Then the adjacency tensor $\mathbb {A}(G)$ of $G$ is a diagonal blocked tensor with diagonal blocks
$\mathbb {A}(G_1),\cdots,\mathbb {A}(G_r)$ (each of which is weakly irreducible).
\end{Exa}

\subsection{A basic property of the third type triangular blocked tensors}
\begin{The}\label{The21} Let the meaning of the notations be as in Definition 2.1，$r\ge 2$, $1\le t\le r-1$, write $k=S_t$ and $I=[k]$.
Then $\mathbb {A}$ is an $(n_1,\cdots,n_r)$-3rdUTB tensor with diagonal blocks $\mathbb {A}_1,\cdots,\mathbb {A}_r$
if and only if the following three conditions are all satisfied:

\vskip 0.1cm

\noindent  {\rm(1)} $\mathbb {A}$ is an $(k,n-k)$-3rdUTB tensor with diagonal blocks $\mathbb {A}[I]$ and $\mathbb {A}[\overline{I}]$.

\vskip 0.1cm

\noindent  {\rm(2)} $\mathbb {A}[I]$ is an $(n_1,\cdots,n_t)$-3rdUTB tensor with diagonal blocks $\mathbb {A}_1,\cdots,\mathbb {A}_t$.

\vskip 0.1cm

\noindent  {\rm(3)} $\mathbb {A}[\overline{I}]$ is an $(n_{t+1},\cdots,n_r)$-3rdUTB tensor with diagonal blocks $\mathbb {A}_{t+1},\cdots,\mathbb {A}_r$.
\end{The}
\begin{proof}  Firstly we write $\mathbb {B}=\mathbb {A}[\overline{I}]$. Thus we have
$$b_{ii_2\cdots i_m}=a_{(i+k)(i_2+k)\cdots (i_m+k)} \quad (i,i_2,\cdots, i_m\in [n-k]).$$

\noindent For $j=1,\cdots,r-t$, we further write $n'_j=n_{t+j}$, and consequently write
$$S'_j=n'_1+\cdots+n'_j, \quad  I'_j=\{S'_{j-1}+1,\cdots,S'_j\} \quad (j=1,\cdots,r-t).$$
Then we have $S'_j=S_{t+j}-S_t=S_{t+j}-k$.

\noindent Necessity.

\vskip 0.1cm

\noindent (1). By Definition \ref{Def23}  we only need to verify the following
$$a_{ii_2\cdots i_m}=0 \quad (\forall i>k, \ \mbox {and} \ \max\{i_2,\cdots,i_m\}\le k). $$
Now assume that $i\in I_j$. Then $i>k=S_t\Longrightarrow j>t$, or $j-1\ge t$. So $\max\{i_2,\cdots,i_m\}\le k=S_t\le S_{j-1}$. By the hypothesis that $\mathbb {A}$ is an $(n_1,\cdots,n_r)$-3rdUTB tensor we have $a_{ii_2\cdots i_m}=0 $. So
$\mathbb {A}$ is an $(k,n-k)$-3rdUTB tensor with diagonal blocks $\mathbb {A}[I]$ and $\mathbb {A}[\overline{I}]$.

\vskip 0.1cm

\noindent (2). Since the condition in Eq.(\ref{eq26}) holds for all $j=2,\cdots,r$ for the tensor $\mathbb {A}$,
it also holds for all $j=2,\cdots,t$ for the tensor $\mathbb {A}[I]$. So $\mathbb {A}[I]$ is an $(n_1,\cdots,n_t)$-3rdUTB tensor with
diagonal blocks $\mathbb {A}_1,\cdots,\mathbb {A}_t$.

\vskip 0.1cm

\noindent (3). By Definition \ref{Def23} we need to show that for any $j=2,\cdots,n-t$, we have
\begin{equation}\label{eq29}
b_{ii_2\cdots i_m}=0 \quad (\forall i\in I'_j, \ \mbox {and} \ \max\{i_2,\cdots,i_m\}\le S'_{j-1}). 
\end{equation}
Now

$i\in I'_j\Longrightarrow S_{t+j-1}-k+1=S'_{j-1}+1\le i\le S_j'=S_{t+j}-k $

\hskip1cm   $\Longrightarrow S_{t+j-1}+1\le i+k\le S_{t+j}$

\hskip1cm   $\Longrightarrow i+k\in I_{t+j}$.

\vskip 0.1cm

\noindent And

$\max\{i_2,\cdots,i_m\}\le S'_{j-1}\Longrightarrow \max\{i_2,\cdots,i_m\}\le S_{t+j-1}-k$

\hskip4.2cm $\Longrightarrow \max\{i_2+k,\cdots,i_m+k\}\le S_{t+j-1}$.

\vskip 0.1cm

\noindent So by the hypothesis that $\mathbb {A}$ is an $(n_1,\cdots,n_r)$-3rdUTB tensor we have
$$b_{ii_2\cdots i_m}=a_{(i+k)(i_2+k)\cdots(i_m+k)}=0.$$
\noindent So $\mathbb {B}=\mathbb {A}[\overline{I}]$ is an $(n_{t+1},\cdots,n_r)$-3rdUTB tensor with diagonal blocks $\mathbb {A}_{t+1},\cdots,\mathbb {A}_r$ (where $(n_{t+1},\cdots,n_r)=(n_1',\cdots,n'_{r-t})$).

\noindent Sufficiency.

\vskip 0.1cm

By Definition \ref{Def23} we need to show that (\ref{eq27}) holds for $j=2,\cdots, r$.
Suppose that $2\le j\le r$, $i\in I_j$, $\max\{i_2,\cdots,i_m\}\le S_{j-1}$. We consider the following two cases.

\vskip 0.1cm

\noindent {\bf Case 1:}   $j\le t$.

Then by $\max\{i_2,\cdots,i_m\}\le S_{j-1}\le S_{t}$ we see that $a_{ii_2\cdots i_m}$ is an entry of $\mathbb {A}[I]$. So by condition (2) we have $a_{ii_2\cdots i_m}=0$.

\vskip 0.1cm

\noindent {\bf Case 2:}  $j\ge t+1$.

Then we consider the following two subcases.

\vskip 0.1cm

\noindent {\bf Subcase 2.1:}  $\max\{i_2,\cdots,i_m\}\le S_t=k$.

Then by $i\in I_j\Longrightarrow i>S_{j-1}\ge S_t=k$, we have $a_{ii_2\cdots i_m}=0$ by condition (1).

\vskip 0.1cm

\noindent {\bf Subcase 2.2:}  $\max\{i_2,\cdots,i_m\}>k$.

Now we still have $$\max\{i_2,\cdots,i_m\}\le S_{j-1}=S'_{j-t-1}+k\Longrightarrow \max\{i_2-k,\cdots,i_m-k\}\le S'_{j-t-1}.$$ Also
$$i\in I_j\Longrightarrow S_{j-1}<i\le S_j\Longrightarrow S'_{j-t-1}<i-k\le S'_{j-t}\Longrightarrow i-k\in  I'_{j-t}.$$
So by condition (3) and using Definition \ref{Def23} for $\mathbb {B}$, we have $b_{(i-k)(i_2-k)\cdots(i_m-k)}=0$, which is equivalent to $a_{ii_2\cdots i_m}=0$.
\end{proof}

In the following we will see that, for the first type triangular blocked tensors, only the necessity part of Theorem \ref{The21} holds
(see Example \ref{Exa24} and Lemma \ref{Lem21} below), and the sufficiency part only holds for the case $t=1$ (see Theorem \ref{The22}).
We will also see that, for the second type triangular blocked tensors, only the sufficiency part of Theorem \ref{The21} holds
(see Example \ref{Exa24} and Lemma \ref{Lem22}  below), and the necessity part only holds for the case $t=r-1$ (see Theorem \ref{The24}).

\begin{Exa}\label{Exa24}
 Take $n=3$, $m=3$ and $\mathbb{A}$ to be a (0,1) tensor of order 3 and dimension 3 with
$$a_{213}=1,$$
and all the other entries of $\mathbb{A}$ are zero. Taking $n_1=n_2=n_3=1$. Then we have

\noindent {\rm(1).}  $\mathbb{A}$ is a (1,1,1)-2ndUTB tensor, but is not a (1,2)-2ndUTB tensor. Because if $\mathbb{A}$ is
a (1,2)-2ndUTB tensor, then it should satisfy the following condition:
$$a_{ii_2i_3}=0 \quad (\forall i\ge 2, \ \min\{i_2,i_3\}\le 1 \  \mbox {and} \ \max\{i_2,i_3\}\le 3). $$
But $a_{213}=1$ means that the above condition does not hold.
This shows that the necessity part (1) of Theorem \ref{The21} does not hold for the second type UTB tensor (by taking $t=1$,
and thus $k=1$ in Theorem \ref{The21}).

\noindent {\rm(2).} By taking $t=2$, and thus $k=2$ in Theorem \ref{The21},
we can check that $\mathbb{A}$ satisfies all the three conditions in Theorem \ref{The21},
but $\mathbb{A}$ is not a (1,1,1)-UTB tensor.  Because if $\mathbb{A}$ is a (1,1,1)-UTB tensor,
then it should satisfy the following condition (taking $j=2$ in (\ref{eq21})):
$$a_{ii_2i_3}=0 \quad (\forall i\in I_2=\{2\}, \ \min\{i_2,i_3\}\le 1 ). $$
But $a_{213}=1$ means that the above condition does not hold. This shows that the sufficiency part of Theorem \ref{The21} does not hold
for the first type UTB tensor.
\end{Exa}

\subsection{Some basic properties and characterizations of the first type triangular blocked tensors}

\hskip.6cm This subsection mainly contains two results.  Firstly we show in Lemma \ref{Lem21} that the necessity part of
Theorem \ref{The21} also holds for the first type triangular blocked tensors.
Secondly we use Lemma \ref{Lem21} to prove Theorem \ref{The22} which gives a necessary and sufficient condition for the first type
$(n_1,\cdots,n_r)$-UTB tensors and will be used later in Section 4.

\begin{Lem}\label{Lem21}
  Let the meaning of the notations be as in Definition \ref{Def21}, $r\ge 2$, $1\le t\le r-1$, write $k=S_t$ and $I=[k]$.
  If $\mathbb {A}$ is an $(n_1,\cdots,n_r)$-UTB tensor with diagonal blocks $\mathbb {A}_1,\cdots,\mathbb {A}_r$, then we have:

\noindent {\rm(1)} $\mathbb {A}$ is an $(k,n-k)$-UTB tensor with diagonal blocks $\mathbb {A}[I]$ and $\mathbb {A}[\overline{I}]$.

\noindent {\rm(2)} $\mathbb {A}[I]$ is an $(n_1,\cdots,n_t)$-UTB tensor with diagonal blocks $\mathbb {A}_1,\cdots,\mathbb {A}_t$.

\noindent {\rm(3)} $\mathbb {A}[\overline{I}]$ is an $(n_{t+1},\cdots,n_r)$-UTB tensor with diagonal blocks
$\mathbb {A}_{t+1},\cdots,\mathbb {A}_r$.
\end{Lem}

Since the proof of Lemma \ref{Lem21}  is similar to the necessity part of Theorem \ref{The21}, we choose to omit its proof.

\vskip 0.1cm

Using Lemma \ref{Lem21}, we can obtain the following characterization for the first type upper triangular blocked tensors. This means that, for the first type upper triangular blocked tensors, the sufficiency part of Theorem \ref{The21} still holds for the case $t=1$.
This theorem will be used later in the proof of Theorem \ref{The41}.

\begin{The}\label{The22}
  Let the meaning of the notations be as in Definition \ref{Def21}, $r\ge 2$.
  Then $\mathbb {A}$ is an $(n_1,\cdots,n_r)$-UTB tensor with diagonal blocks $\mathbb {A}_1,\cdots,\mathbb {A}_r$ if and only if
  $\mathbb {A}$ satisfies the following two conditions:

\noindent {\rm(1)} $\mathbb {A}$ is an $(n_1,n-n_1)$-UTB tensor with diagonal blocks $\mathbb {A}_1$ and $\mathbb {A}[\overline{I_1}]$.

\noindent {\rm(2)} $\mathbb {A}[\overline{I_1}]$ is an $(n_{2},\cdots,n_r)$-UTB tensor with diagonal blocks $\mathbb {A}_{2},\cdots,\mathbb {A}_r$.
\end{The}
\begin{proof}
The necessity part follows from Lemma \ref{Lem21}. Now we prove the sufficiency part. Similarly as in Lemma \ref{Lem21},
we write $\mathbb {B}=\mathbb {A}[\overline{I_1}]$. Thus we have:
$$b_{ii_2\cdots i_m}=a_{(i+n_1)(i_2+n_1)\cdots (i_m+n_1)} \quad (i,i_2,\cdots, i_m\in [n-n_1]).$$
For $j=1,\cdots,r-1$, we further write $n'_j=n_{j+1}$, and consequently write
$$S'_j=n'_1+\cdots+n'_j, \quad  I'_j=\{S'_{j-1}+1,\cdots,S'_j\} \quad (j=1,\cdots,r-1).$$
Then we have $S'_j=S_{j+1}-n_1$.

By Definition \ref{Def21} we need to show that (\ref{eq21}) holds for $j=2,\cdots, r$.
Suppose that $2\le j\le r$, $i\in I_j$ and $\min\{i_2,\cdots,i_m\}\le S_{j-1}$ (say, $i_2=\min\{i_2,\cdots,i_m\}$).
Then we consider the following two cases:

\noindent {\bf Case 1:}  $i_2\le n_1$.

Then by $i\in I_j\Longrightarrow i>S_{j-1}\ge S_1=n_1$, we have $a_{ii_2\cdots i_m}=0$ by Example \ref{Exa21} and condition (1).

\noindent {\bf Case 2:} $i_2>n_1$.

Then by hypothesis we have $i_2\le S_{j-1}=S'_{j-2}+n_1\Longrightarrow i_2-n_1\le S'_{j-2}$. Also
$$i\in I_j\Longrightarrow S_{j-1}+1\le i\le S_j\Longrightarrow S'_{j-2}+1\le i-n_1\le S'_{j-1}\Longrightarrow i-n_1\in I'_{j-1}.$$
So by condition (2) we have $b_{(i-n_1)(i_2-n_1)\cdots(i_m-n_1)}=0$, which is equivalent to $a_{ii_2\cdots i_m}=0$.
\end{proof}

Using Theorem \ref{The22}, we can obtain the following equivalent definition for the $(n_1,\cdots,n_r)$-UTB tensor
(we omit its proof since this result will not be used in the remaining of this paper.

\begin{The}\label{The23}   Let the meaning of the notations be as in Definition \ref{Def21}, $r\ge 2$.
Then $\mathbb{A}$ is an $(n_1,\cdots,n_r)$-UTB tensor with diagonal blocks $\mathbb{A}_1,\cdots,\mathbb{A}_r$ if and only if
for each $i=1,\cdots,r-1$, $\mathbb{A}[I_i\cup\cdots\cup I_r]$ is an $(n_i,n_{i+1}+\cdots+n_r)$-UTB tensor with
diagonal blocks $\mathbb{A}_i$ and $\mathbb{A}[I_{i+1}\cup\cdots\cup I_r]$.
\end{The}

\subsection{Some basic properties and characterizations of the second type triangular blocked tensors}
\hskip.6cm This subsection mainly contains three results.  Firstly we show in Lemma \ref{Lem22}  that the sufficiency part of
Theorem \ref{The21} also holds for the second type triangular blocked tensors. Secondly we use Lemma \ref{Lem22} to prove
Theorem \ref{The24} which gives a necessary and sufficient condition for the second type $(n_1,\cdots,n_r)$-UTB
tensors and will be used later in Section 3, 4 and 6. Thirdly we show in Theorem \ref{The25} that the definition of the second
type UTB tensors in this paper is essentially equivalent to a definition in Definition 4.3 of \cite{2013SSZ}.

\begin{Lem}\label{Lem22} Let the meaning of the notations be as in Definition \ref{Def21}, $r\ge 2$, $1\le t\le r-1$,
write $k=S_t$ and $I=[k]$. If $\mathbb {A}$ satisfies the following three conditions:

\noindent {\rm(1)} $\mathbb {A}$ is an $(k,n-k)$-2ndUTB tensor with diagonal blocks $\mathbb {A}[I]$ and $\mathbb {A}[\overline{I}]$.

\noindent {\rm(2)} $\mathbb {A}[I]$ is an $(n_1,\cdots,n_t)$-2ndUTB tensor with diagonal blocks $\mathbb {A}_1,\cdots,\mathbb {A}_t$.

\noindent {\rm(3)} $\mathbb {A}[\overline{I}]$ is an $(n_{t+1},\cdots,n_r)$-2ndUTB tensor with diagonal blocks
$\mathbb {A}_{t+1},\cdots,\mathbb {A}_r$.

\noindent Then $\mathbb {A}$ is an $(n_1,\cdots,n_r)$-2ndUTB tensor with diagonal blocks $\mathbb {A}_1,\cdots,\mathbb {A}_r$.
\end{Lem}
\begin{proof} Similarly as in Lemma \ref{Lem21}, we write $\mathbb {B}=\mathbb {A}[\overline{I}]$. Thus we have
$$b_{ii_2\cdots i_m}=a_{(i+k)(i_2+k)\cdots (i_m+k)} \quad (i,i_2,\cdots, i_m\in [n-k]).$$
For $j=1,\cdots,r-t$, we further write $n'_j=n_{t+j}$, and consequently write
$$S'_j=n'_1+\cdots+n'_j, \quad  I'_j=\{S'_{j-1}+1,\cdots,S'_j\} \quad (j=1,\cdots,r-t).$$
Then we have $S'_j=S_{t+j}-S_t=S_{t+j}-k$.

By Definition \ref{Def22}  we need to show that (\ref{eq24}) holds for $j=2,\cdots, r$.
Suppose that $2\le j\le r$, $i\in I_j$, $\min\{i_2,\cdots,i_m\}\le S_{j-1}$ (say, $i_2=\min\{i_2,\cdots,i_m\}$),
and $\max\{i_2,\cdots,i_m\}\le S_{j}$. We consider the following two cases.

\noindent {\bf Case 1:} $j\le t$.

Then by $\max\{i_2,\cdots,i_m\}\le S_{j}\le S_{t}$ we see that $a_{ii_2\cdots i_m}$
is an entry of $\mathbb {A}[I]$. So by condition (2) we have $a_{ii_2\cdots i_m}=0$.

\noindent {\bf Case 2:}  $j\ge t+1$.

Then for $i_2=\min\{i_2,\cdots,i_m\}\le S_{j-1}$ and $\max\{i_2,\cdots,i_m\}\le S_{j}$, we consider the following two subcases.

\noindent {\bf Subcase 2.1:} $i_2\le S_t=k$.

Then by $i\in I_j\Longrightarrow i>S_{j-1}\ge S_t=k$, we have $a_{ii_2\cdots i_m}=0$ by Example \ref{Exa21}nd condition (1).

\noindent {\bf Subcase 2.2:} $i_2>k$.

But we still have $i_2\le S_{j-1}=S'_{j-t-1}+k\Longrightarrow i_2-k\le S'_{j-t-1}$. Also
$$i\in I_j\Longrightarrow S_{j-1}<i\le S_j\Longrightarrow S'_{j-t-1}<i-k\le S'_{j-t}\Longrightarrow i-k\in  I'_{j-t}.$$
Furthermore, we have $\max\{i_2,\cdots,i_m\}\le S_{j}\Longrightarrow \max\{i_2-k,\cdots,i_m-k\}\le S'_{j-t}$.
So by condition (3) and using Definition \ref{Def22} for $\mathbb {B}$, we have $b_{(i-k)(i_2-k)\cdots(i_m-k)}=0$,
which is equivalent to $a_{ii_2\cdots i_m}=0$.
\end{proof}

Using Lemma \ref{Lem22}, we can obtain the following characterization for the second type upper triangular blocked tensors. This means that, for the second type upper triangular blocked tensors, the necessity part of Theorem \ref{The21} still holds for the case $t=r-1$. Also, the following Theorem \ref{The24} will be used later in the proofs of Theorems \ref{The33}, \ref{The42} and \ref{The63}.

\begin{The}\label{The24} Let the meaning of the notations be as in Definition \ref{Def21}, $r\ge 2$.
Then $\mathbb{A}$ is an $(n_1,\cdots,n_r)$-2ndUTB tensor with diagonal blocks $\mathbb{A}_1,\cdots,\mathbb{A}_r$ if and only if
$\mathbb{A}$ satisfies the following two conditions:

\noindent {\rm(1)} $\mathbb {A}$ is an $(n-n_r,n_r)$-2ndUTB tensor with diagonal blocks $\mathbb{A}[I_1\cup\cdots\cup I_{r-1}]$ and $\mathbb{A}_r$.

\noindent {\rm(2)} $\mathbb{A}[I_1\cup\cdots\cup I_{r-1}]$ is an $(n_{1},\cdots,n_{r-1})$-2ndUTB tensor with diagonal blocks
$\mathbb {A}_{1},\cdots,\mathbb {A}_{r-1}$.
\end{The}
\begin{proof} The sufficiency part follows from Lemma \ref{Lem22}. Now we prove the necessity part.

\noindent {\rm(1).} Write $k=n-n_r=n_1+\cdots+n_{r-1}=S_{r-1}$. By Example \ref{Exa21} we only need to verify the following
$$a_{ii_2\cdots i_m}=0 \quad (\forall i>k, \ \mbox {and} \ \min\{i_2,\cdots,i_m\}\le k).$$ 

Now assume that $i\in I_j$. Then $S_j\ge i>k=S_{r-1}\Longrightarrow j>r-1\Longrightarrow j=r$. In this case,
$\max\{i_2,\cdots,i_m\}\le S_j=S_r=n$ holds automatically. So by the hypothesis that $\mathbb {A}$ is an $(n_1,\cdots,n_r)$-2ndUTB
tensor and Definition 2.2 and the fact that $k=S_{r-1}$ we have $a_{ii_2\cdots i_m}=0$.

\noindent {\rm(2).} Since the condition in Eq.(\ref{eq24}) holds for all $j=2,\cdots,r$ for the tensor $\mathbb {A}$,
it also holds for all $j=2,\cdots,r-1$ for the tensor $\mathbb{A}[I_1\cup\cdots\cup I_{r-1}]$. So the conclusion follows from Definition \ref{Def22}.
\end{proof}

The following theorem shows that our definition for the $(n_1,\cdots,n_r)$-2ndUTB tensors is essentially equivalent to Definition 4.3 in
\cite{2013SSZ}. We would like to point out that this theorem will not be used later in this paper.

\begin{The}\label{The25}  Let the meaning of the notations be as in Definition \ref{Def21}, $r\ge 2$.
Then $\mathbb{A}$ is an $(n_1,\cdots,n_r)$-2ndUTB tensor with diagonal blocks $\mathbb{A}_1,\cdots,\mathbb{A}_r$ if and only if
for each $i=2,\cdots,r$, $\mathbb{A}[I_1\cup\cdots\cup I_i]$ is an $((n_1+\cdots+n_{i-1}),n_{i})$-2ndUTB tensor with diagonal
blocks  $\mathbb{A}[I_{1}\cup\cdots\cup I_{i-1}]$ and $\mathbb{A}_i$.
\end{The}
\begin{proof}
Necessity. The case $i=r$ follows from the necessity part (1) of Theorem \ref{The24}. Now we assume $1\le i\le r-1$.
By the necessity part (2) of Theorem \ref{The24}, $\mathbb{A}[I_{1}\cup\cdots\cup I_{r-1}]$ is an $(n_1,\cdots,n_{r-1})$-2ndUTB tensor with
diagonal blocks $\mathbb{A}_1,\cdots,\mathbb{A}_{r-1}$. Thus by induction on $r$ for the tensor $\mathbb{A}[I_{1}\cup\cdots\cup I_{r-1}]$,
we obtain the desired result.

Sufficiency. We use induction on $r$. The case $r=2$ is trivial. Now we assume that $r\ge 3$. By induction and the case $i=2,\cdots,r-1$ of the
hypothesis,  $\mathbb{A}[I_1\cup\cdots\cup I_{r-1}]$ is an $(n_1,\cdots,n_{r-1})$-2ndUTB tensor with diagonal blocks
$\mathbb{A}_1,\cdots,\mathbb{A}_{r-1}$. By the case $i=r$ of the hypothesis, and using the sufficiency part of the Theorem \ref{The24},
we conclude that $\mathbb{A}$ is an $(n_1,\cdots,n_r)$-2ndUTB tensor with diagonal blocks $\mathbb{A}_1,\cdots,\mathbb{A}_r$.
\end{proof}

Comparing Theorem \ref{The25} with Definition 4.3 and Remark 4.1 in \cite{2013SSZ} we can see that, if we transfer Definition 4.3 of \cite{2013SSZ} from the form of lower triangular blocked tensors to its equivalent form of upper triangular blocked tensors, then that type of $(n_1,\cdots,n_r)$-lower triangular blocked tensors is actually equivalent to the second type $(n_1,\cdots,n_r)$-upper triangular blocked tensors defined in this paper.

\section{The determinants, characteristic polynomials and spectra of the triangular blocked tensors}
\hskip.6cm In this section, we will prove some formulas for the determinants and characteristic polynomials of the first and second type
UTB tensors, and give an example to show that these formulas do not hold for the third type UTB tensors.

The following result for the special case $r=2$ was proved by Hu et al in \cite{2013HHLQ}. For the convenience and self-containedness,
here we give an outline of a simplified proof.

\begin{The}\label{The31}{\rm(\cite{2013HHLQ})}
Let $\mathbb{A}$ be an order $m$ dimension $n$ and $(k,n-k)$-UTB (also 2ndUTB by Remark \ref{Rem21}) tensor with diagonal blocks
$\mathbb {A}_1$ and $\mathbb{A}_2$. Then
\begin{equation}\label{eq31}\det \mathbb {A}= (\det \mathbb {A}_1)^{(m-1)^{n-k}} ( \det \mathbb {A}_2)^{(m-1)^{k}} \end{equation} 
\end{The}
\begin{proof} {\bf Step 1:} To show that $\det \mathbb {A}=0\Longrightarrow (\det \mathbb{A}_1)( \det \mathbb{A}_2)=0$. Suppose not, then

$$ (\det\mathbb {A}_1)(\det\mathbb {A}_2)\ne 0 \Longrightarrow  \left \{   \begin{array}{c}
                         \det\mathbb {A}_2\ne 0 \Longrightarrow  \mathbb {A}_2z=0 \  \  \mbox  {has only zero solution}          \\
                         \det\mathbb {A}_1\ne 0 \Longrightarrow  \mathbb {A}_1y=0 \  \  \mbox  {has only zero solution}          \\
                               \end{array}
                          \right. $$

\hskip4.5cm $\stackrel{(\ast)}{\Longrightarrow} \mathbb {A}x=\mathbb {A}\left(
   \begin{array}{c}
     y\\
     z \\
   \end{array}
 \right)=0$ has only zero solution

\hskip4.5cm  $\Longrightarrow \det\mathbb {A}\ne 0$.

\noindent where the implication $(\ast)$ holds because: $\mathbb {A}$ is an $(k,n-k)$-UTB tensor with diagonal blocks $\mathbb {A}_1$ and $\mathbb{A}_2$ implies that
$$\mathbb{A}\left(
   \begin{array}{c}
     y\\
     z \\
   \end{array}
 \right)=\left(
   \begin{array}{c}
     *\\
     \mathbb{A}_2z \\
   \end{array}
 \right)  \quad \mbox {and} \quad  \mathbb{A}\left(
   \begin{array}{c}
     y\\
     0 \\
   \end{array}
 \right)=\left(
   \begin{array}{c}
     \mathbb{A}_1y\\
     0 \\
   \end{array}
 \right).  $$

\noindent {\bf Step 2:} By using Hilbert's Nullstellensatz (\cite{1998}) and the irreducibility of the polynomials
$\det\mathbb {A}_1$ and $\det\mathbb {A}_2$ (\cite{2013HHLQ}), we have：
Every irreducible factor of $\det\mathbb {A}$ is an irreducible factor of  $(\det\mathbb {A}_1)(\det\mathbb {A}_2)$.

Thus there exist constants  $c, r_1,r_2$ (independent of the tensor $\mathbb {A}$) such that
\begin{equation}\label{eq32}
\det\mathbb {A}=c\cdot (\det\mathbb {A}_1)^{r_1}(\det\mathbb {A}_2)^{r_2}.\end{equation} 

\noindent{\bf Step 3:}  Take $\mathbb {A}=\mathbb {I}$ to be the identity tensor, then both $\mathbb {A}_1$ and $\mathbb {A}_2$
are also the identity tensors, thus we obtain $c=1$ from (\ref{eq32}).

\noindent {\bf Step 4 }(To further determine $r_1, r_2$){\bf:} Take $\mathbb{A}$ to be the diagonal tensor with the first $k$ diagonal entries $a$,
and all the rest diagonal entries $b$, then we have:
\begin{equation}\label{eq33}\det\mathbb {A}=a^{k(m-1)^{n-1}}b^{(n-k)(m-1)^{n-1}},\end{equation} 
and
\begin{equation}\label{eq34}\det\mathbb {A}_1=a^{k(m-1)^{k-1}},\qquad  det\mathbb {A}_2=b^{(n-k)(m-1)^{n-k-1}}.\end{equation} 
Substituting (\ref{eq33}), (\ref{eq34}) and $c=1$ into (\ref{eq32}), we obtain
\begin{equation}\label{eq35}a^{k(m-1)^{n-1}}b^{(n-k)(m-1)^{n-1}}=a^{r_1k(m-1)^{k-1}}b^{r_2(n-k)(m-1)^{n-k-1}}.\end{equation} 

\noindent Comparing the exponents of $a$ and $b$ of the both sides of (\ref{eq35}), we obtain $r_1=(m-1)^{n-k}, \ r_2=(m-1)^{k}.$
\end{proof}

The following example shows that Theorem \ref{The31} does not hold for the third type $(k,n-k)$-UTB tensors.

\begin{Exa}\label{Exa31}
 Take a $(0,1)$ tensor $\mathbb{A}$ of order $m=3$ and dimension $n=2$ with
$$a_{111}=a_{122}=1, \  a_{112}=a_{121}=0 \quad and \quad a_{211}=a_{222}=0, \  a_{212}= a_{221}=1.$$
Then $\mathbb{A}$ is a (1,1)-3rdUTB tensor (since $a_{211}=0$) with the two diagonal blocks $\mathbb{A}_1=a_{111}=1$ and
$\mathbb{A}_2=a_{222}=0$. In this case, the right hand side of (\ref{eq31}) is zero since $\det \mathbb {A}_2=0$.

\vskip 0.1cm

On the other hand, we now show that $\det \mathbb {A}\ne 0$. By the definition we can verify that $\mathbb {A}x=0$ is the following system of equations:

$$\left \{   \begin{array}{c}
                         x_1^2+x_2^2=0;          \\
                         2x_1x_2=0,         \\
                               \end{array}
                          \right. $$
which has only a zero solution. Thus we have $\det \mathbb {A}\ne 0$ by the definition of the tensor determinants,
and so (\ref{eq31}) does not hold for this $(1,1)$-3rdUTB tensor $\mathbb{A}$.
\end{Exa}

Now we show that Theorem \ref{The31} also holds for lower triangular blocked tensors.

\begin{The}\label{The32}
Theorem \ref{The31} also holds for lower triangular blocked tensors.
\end{The}
\begin{proof} Write $k'=n-k$. Then by (5) of Remark \ref{Rem24} we have that:

\hskip.1cm $\mathbb{A}$ is an order $m$ dimension $n$ and $(k,n-k)$-LTB tensor with diagonal blocks $\mathbb {A}_1$ and $\mathbb{A}_2$

\noindent $\Longrightarrow \mathbb {A}$ is permutation similar to a $(k',n-k')$-UTB tensor $\mathbb {B}$ with diagonal blocks $\mathbb {A}_2$ and $\mathbb{A}_1$.

\noindent Thus by \cite{2013S}  (similar tensors have the same determinants) and using Theorem \ref{The31}  for $\mathbb {B}$ we have:
$$\det \mathbb{A}=\det \mathbb{B}=(\det \mathbb {A}_2)^{(m-1)^{n-k'}} ( \det \mathbb {A}_1)^{(m-1)^{k'}}
=(\det \mathbb {A}_1)^{(m-1)^{n-k}} ( \det \mathbb {A}_2)^{(m-1)^{k}}. $$
\end{proof}

By using Theorem \ref{The31} and Lemma \ref{Lem21}, we can use induction to obtain the formula for determinants of
general $(n_1,\cdots,n_r)$-UTB (and 2ndUTB) tensors as follows.  (The case of the second type also appeared in Shao et al \cite{2013SSZ}.)

\begin{The}\label{The33} Let $\mathbb{A}$ be an order $m$ dimension $n$ and $(n_1,\cdots,n_r)$-UTB (or 2ndUTB) tensor
with diagonal blocks $\mathbb{A}_1,\cdots,\mathbb{A}_r$. Then we have:
\begin{equation}\label{eq36}\det\mathbb {A}=\prod_{i=1}^r (\det\mathbb {A}_i)^{(m-1)^{n-n_i}}.\end{equation}
\end{The}

\begin{proof} We use induction on $r$. If $r=2$, the result is just Theorem \ref{The31}. Now we assume $r\ge 3$.
Write $\mathbb{B}=\mathbb{A}[I_1\cup\cdots\cup I_{r-1}]$, then by Lemma \ref{Lem21}  and Theorem \ref{The24}  we know that $\mathbb{A}$
is an $(n-n_r,n_r)$-UTB (or 2ndUTB) tensor with diagonal blocks $\mathbb{B}$ and $\mathbb{A}_r$. Thus by Theorem \ref{The31}  we have
\begin{equation}\label{eq37}\det \mathbb {A}= (\det \mathbb {B})^{(m-1)^{n_r}} ( \det \mathbb {A}_r)^{(m-1)^{n-n_r}} .\end{equation}

\noindent Also by Lemma \ref{Lem21}  and Theorem \ref{The24} we know that $\mathbb{B}$ is an $(n_1,\cdots,n_{r-1})$-UTB (or 2ndUTB) tensor
with diagonal blocks $\mathbb{A}_1,\cdots,\mathbb{A}_{r-1}$. So by induction we have
\begin{equation}\label{eq38}\det\mathbb {B}=\prod_{i=1}^{r-1} (\det\mathbb {A}_i)^{(m-1)^{n-n_r-n_i}}.\end{equation}

\noindent Substituting (\ref{eq38}) into (\ref{eq37}), we obtain (\ref{eq36}).
\end{proof}

\begin{The}\label{The34} Let $\mathbb{A}$ be an order $m$ dimension $n$ and $(n_1,\cdots,n_r)$-UTB (or 2ndUTB) tensor with diagonal blocks
$\mathbb{A}_1,\cdots,\mathbb{A}_r$. Then we have the following relations for the characteristic polynomial and spectrum of $\mathbb{A}$:

\begin{equation}\label{eq39}\phi_{\mathbb {A}}(\lambda)=\prod_{i=1}^r (\phi_{\mathbb {A}_i}(\lambda))^{(m-1)^{n-n_i}},\end{equation}

\noindent and
\begin{equation}\label{eq310}Spec(\mathbb {A})=\bigcup_{i=1}^r (Spec(\mathbb {A}_i))^{(m-1)^{n-n_i}},\end{equation}

\noindent where the notation $S^k$ of a multi-set $S$ denotes the repetition of $k$ times of the set $S$, and
\begin{equation}\label{eq311} \rho(\mathbb {A})=\max \{\rho(\mathbb {A}_i) \ | \ i=1,\cdots,r\}. \end{equation}
\end{The}

\begin{Cor}\label{Cor31}   Let $G$ be a $k$-uniform hypergraph with the $r$ connected components $G_1,\cdots,G_r$. Then we have
\begin{equation}\label{eq312}Spec(\mathbb {A}(G))=\bigcup_{i=1}^r (Spec(\mathbb {A}(G_i)))^{(m-1)^{n-n_i}}, \end{equation}

\noindent and
\begin{equation}\label{eq313}\rho(\mathbb {A}(G))=\max \{\rho(\mathbb {A}(G_i)) \ | \ i=1,\cdots,r\}. \end{equation}
\end{Cor}
\begin{proof} By Example \ref{Exa23}, the adjacency tensor $\mathbb {A}(G)$ of $G$ is a diagonal blocked tensor with diagonal blocks
$\mathbb {A}(G_1),\cdots,\mathbb {A}(G_r)$. \end{proof}

\section{The products of the triangular blocked tensors}

\hskip.6cm In this section we will show that the product of any two $(n_1,\cdots,n_r)$ upper (or lower) triangular blocked tensors
of the first or second or third type is still an $(n_1,\cdots,n_r)$ upper (or lower) triangular blocked tensor of the same type.
We first prove the case $r=2$ in Lemma \ref{Lem41}  (for the first and second types) and Lemma \ref{Lem42} (for the third type).

The first result (but not the second result about the two diagonal blocks of $\mathbb {A}\mathbb {B}$) of the following lemma
can be found in \cite{2013SSZ} (Proposition 4.1). Recall that in Remark \ref{Rem21}, we have mentioned that in the case $r=2$,
the $(n_1,n_2)$-UTB tensors and $(n_1,n_2)$-2ndUTB tensors are the same.

\begin{Lem}\label{Lem41}   Let $\mathbb{A}$ and $\mathbb{B}$ be order $m$ and order $k$, dimension $n$ and $(p,n-p)$-UTB
(also 2ndUTB by Remark  \ref{Rem21}) tensor with diagonal blocks $\mathbb {A}_1$, $\mathbb{A}_2$ and $\mathbb {B}_1$,
$\mathbb{B}_2$, respectively. Then their product $\mathbb {A}\mathbb {B}$ is also a $(p,n-p)$-UTB tensor with the two diagonal blocks
$\mathbb {A}_1\mathbb {B}_1$ and $\mathbb{A}_2\mathbb {B}_2$.
\end{Lem}

\begin{proof} Write $\mathbb{C}= \mathbb {A}\mathbb {B}$. Then by the definition of the tensor product we have (see (\ref{eq13})):
\begin{equation}\label{eq41}
c_{i\alpha_1\cdots \alpha_{m-1}}=\sum_{i_2,\cdots ,i_m=1}^n a_{ii_2\cdots i_m}b_{i_2\alpha_1}\cdots b_{i_m\alpha_{m-1}}
\qquad  (i\in [n], \  \alpha_1, \cdots ,\alpha_{m-1}\in [n]^{k-1}). \end{equation} 

\noindent The proof of the first result that $\mathbb{C}$ is a $(p,n-p)$-UTB tensor can be found in \cite{2013SSZ} (Proposition 4.1).

Now we show that the diagonal blocks $\mathbb {C}_1$ and $\mathbb {C}_2$ of $\mathbb {A}\mathbb {B}$ are $\mathbb {A}_1\mathbb {B}_1$
and $\mathbb{A}_2\mathbb {B}_2$.
For the first diagonal block $\mathbb {C}_1$, let $i$ and all the subscripts in $\alpha_1, \cdots ,\alpha_{m-1}$ are all $\le p$.
In this case, if all $b_{i_j\alpha_{j-1}}\ne 0$, then we must have all $i_j\le p \ (j=2,\cdots,m)$ since $\mathbb {B}$ is $(p,n-p)$-UTB.
Thus we have
$$ c_{i\alpha_1\cdots \alpha_{m-1}}=\sum_{i_2,\cdots ,i_m=1}^n a_{ii_2\cdots i_m}b_{i_2\alpha_1}\cdots b_{i_m\alpha_{m-1}}
=\sum_{i_2,\cdots ,i_m=1}^p a_{ii_2\cdots i_m}b_{i_2\alpha_1}\cdots b_{i_m\alpha_{m-1}}
=(\mathbb {A}_1\mathbb {B}_1)_{i\alpha_1\cdots \alpha_{m-1}},$$
so we obtain $\mathbb {C}_1=\mathbb {A}_1\mathbb {B}_1$.

\vskip 0.1cm

Now we consider the second diagonal block $\mathbb {C}_2$. Let $i$ and all the subscripts in $\alpha_1, \cdots ,\alpha_{m-1}$ are all $> p$.
In this case, if $\min\{i_2,\cdots ,i_m\}\le p$, then we will have $a_{ii_2\cdots i_m}=0$ since $\mathbb {A}$ is $(p,n-p)$-UTB. Thus we have
$$c_{i\alpha_1\cdots \alpha_{m-1}}=\sum_{i_2,\cdots ,i_m=1}^n a_{ii_2\cdots i_m}b_{i_2\alpha_1}\cdots b_{i_m\alpha_{m-1}}
=\sum_{i_2,\cdots ,i_m=p+1}^n a_{ii_2\cdots i_m}b_{i_2\alpha_1}\cdots b_{i_m\alpha_{m-1}}.$$

Let $j=i-p$, $j_t=i_t-p$, and $\beta_j$ be the sequence of the subscripts obtained by subtracting $p$ from all the subscripts in $\alpha_j$.
Then we have
$$\begin{aligned}
(\mathbb {A}_2\mathbb {B}_2)_{j\beta_1\cdots \beta_{m-1}} &= \sum_{j_2,\cdots ,j_m=1}^{n-p} (\mathbb {A}_2)_{jj_2\cdots j_m}(\mathbb {B}_2)_{j_2\beta_1}\cdots (\mathbb {B}_2)_{j_m\beta_{m-1}}\\
&=\sum_{i_2,\cdots ,i_m=p+1}^n a_{ii_2\cdots i_m}b_{i_2\alpha_1}\cdots b_{i_m\alpha_{m-1}}
=c_{i\alpha_1\cdots \alpha_{m-1}}.
\end{aligned}
$$
So we obtain $\mathbb {C}_2=\mathbb {A}_2\mathbb {B}_2$.
\end{proof}

Now we show that Lemma \ref{Lem41} is also true for the 3rdUTB tensors.

\begin{Lem}\label{Lem42}    Lemma \ref{Lem41} is also true for the 3rdUTB tensors.\end{Lem}
\begin{proof} The proof is similar to the proof of Lemma \ref{Lem41}. We only give the proof starting from the following step.

We first prove that $\mathbb{C}$ is a $(p,n-p)$-3rdUTB tensor. By Definition \ref{Def23} we need to show that $\mathbb{C}$
satisfies the following condition:
\begin{equation}\label{eq43}c_{i\alpha_1\cdots \alpha_{m-1}}=0 \quad (\forall i>p, \ \mbox {and all indices in
$\alpha_1, \cdots ,\alpha_{m-1}$ are $\le p$}).\end{equation}

Now we consider the following two cases to show that every term in the summation of (\ref{eq41}) is zero.

\noindent {\bf Case 1:} $\max\{i_2,\cdots,i_m\}\le p$.

Then $a_{ii_2\cdots i_m}=0$ since $i>p$ and $\mathbb {A}$ is $(p,n-p)$-3rdUTB.

\noindent {\bf Case 2:} $i_j>p$ for some $j\in \{2,\cdots,m\}$.

Then $b_{i_{j}\alpha_{j-1}}=0$, since all indices in $\alpha_{j-1}$ are $\le p$, and $\mathbb {B}$ is $(p,n-p)$-3rdUTB.

Combining these two cases, we see that every term in the summation of (\ref{eq41}) is zero. So (\ref{eq43}) holds and thus $\mathbb{C}$
is also a $(p,n-p)$-3rdUTB tensor.

The proof of the second conclusion is similar to that of Lemma \ref{Lem41}.
\end{proof}

The following theorem gives the results on the products of the first type or third type UTB tensors.

\begin{The}\label{The41}
   Let $\mathbb{A}$ and $\mathbb{B}$ be order $m$ and order $k$, dimension $n$ and $(n_1,\cdots,n_r)$-UTB (or 3rdUTB) tensor with diagonal
   blocks $\mathbb {A}_1,\cdots, \mathbb{A}_r$ and $\mathbb {B}_1,\cdots, \mathbb{B}_r$, respectively.
   Then their product $\mathbb {A}\mathbb {B}$ is also an $(n_1,\cdots,n_r)$-UTB (or 3rdUTB) tensor with diagonal blocks
   $\mathbb {A}_1\mathbb {B}_1,\cdots, \mathbb{A}_r\mathbb {B}_r$.
\end{The}
\begin{proof} The proof for the first type and third type UTB tensors are the same. So here we only give the proof for the first type UTB tensors.

We use Theorem \ref{The22}, Lemma \ref{Lem41} and induction on $r$. If $r=2$, this is just Lemma \ref{Lem41}. So we assume $r\ge 3$.
Write $\mathbb {C}=\mathbb {A}[I_2\cup\cdots\cup I_r]$ and $\mathbb {D}=\mathbb {B}[I_2\cup\cdots\cup I_r]$. Then by Theorem \ref{The22} we know that:

\noindent {\rm(i)} $\mathbb {A}$ is an $(n_1,n-n_1)$-UTB tensor with diagonal blocks $\mathbb {A}_1$ and $\mathbb{C}$.

\noindent {\rm(ii)} $\mathbb {B}$ is an $(n_1,n-n_1)$-UTB tensor with diagonal blocks $\mathbb {B}_1$ and $\mathbb{D}$.

\noindent {\rm(iii)} $\mathbb {C}$ and $\mathbb {D}$ are $(n_2,\cdots,n_r)$-UTB tensors with diagonal blocks $\mathbb {A}_2,\cdots, \mathbb{A}_r$
 and $\mathbb {B}_2,\cdots, \mathbb{B}_r$, respectively.

Thus by (i), (ii) and Lemma \ref{Lem41} we know that $\mathbb {A}\mathbb {B}$ is an $(n_1,n-n_1)$-UTB tensor with diagonal blocks
$\mathbb {A}_1\mathbb {B}_1$ and $\mathbb{C}\mathbb {D}$.

By using (iii) and induction on $r$, we also know that $\mathbb{C}\mathbb {D}$ is an $(n_2,\cdots,n_r)$-UTB tensor with diagonal blocks
$\mathbb {A}_2\mathbb {B}_2,\cdots, \mathbb{A}_r\mathbb {B}_r$.

Combining these two results with Theorem \ref{The22}, we conclude that $\mathbb {A}\mathbb {B}$ is also an $(n_1,\cdots,n_r)$-UTB tensor
with diagonal blocks $\mathbb {A}_1\mathbb {B}_1,\cdots, \mathbb{A}_r\mathbb {B}_r$.
\end{proof}

Theorem \ref{The41} for the first and third types used the induction from top to bottom. But the following Theorem \ref{The42}
for the second type needs to use the induction from bottom to top.

\begin{The}\label{The42}  Let $\mathbb{A}$ and $\mathbb{B}$ be order $m$ and order $k$, dimension $n$ and $(n_1,\cdots,n_r)$-2ndUTB tensor with
diagonal blocks $\mathbb {A}_1,\cdots, \mathbb{A}_r$ and $\mathbb {B}_1,\cdots, \mathbb{B}_r$, respectively.
Then their product $\mathbb {A}\mathbb {B}$ is also an $(n_1,\cdots,n_r)$-2ndUTB tensor with diagonal blocks
$\mathbb {A}_1\mathbb {B}_1,\cdots, \mathbb{A}_r\mathbb {B}_r$.\end{The}

\begin{proof} We use Theorem \ref{The24}, Lemma \ref{Lem41} and induction on $r$. If $r=2$, this is just Lemma \ref{Lem41}. So we assume $r\ge 3$.
Write $\mathbb {C}=\mathbb {A}[I_1\cup\cdots\cup I_{r-1}]$ and $\mathbb {D}=\mathbb {B}[I_1\cup\cdots\cup I_{r-1}]$.
Then by Theorem \ref{The24} we know that:

\noindent {\rm(i)} $\mathbb {A}$ is an $(n-n_r,n_r)$-2ndUTB tensor with diagonal blocks $\mathbb {C}$ and $\mathbb{A}_r$.

\noindent {\rm(ii)} $\mathbb {B}$ is an $(n-n_r,n_r)$-2ndUTB tensor with diagonal blocks $\mathbb {D}$ and $\mathbb{B}_r$.

\noindent {\rm(iii)} $\mathbb {C}$ and $\mathbb {D}$ are $(n_1,\cdots,n_{r-1})$-2ndUTB tensors with diagonal blocks
$\mathbb {A}_1,\cdots, \mathbb{A}_{r-1}$ and $\mathbb {B}_1,\cdots, \mathbb{B}_{r-1}$, respectively.

Thus by (i), (ii) and Lemma \ref{Lem41} we know that $\mathbb {A}\mathbb {B}$ is an $(n-n_r,n_r)$-2ndUTB tensor with diagonal blocks
$\mathbb{C}\mathbb {D}$ and $\mathbb {A}_r\mathbb {B}_r$.

By using (iii) and induction on $r$, we also know that $\mathbb{C}\mathbb {D}$ is an $(n_1,\cdots,n_{r-1})$-2ndUTB tensor with diagonal blocks
$\mathbb {A}_1\mathbb {B}_1,\cdots, \mathbb{A}_{r-1}\mathbb {B}_{r-1}$.

Combining these two results with Theorem \ref{The24}, we conclude that $\mathbb {A}\mathbb {B}$ is also an $(n_1,\cdots,n_r)$-2ndUTB tensor
with diagonal blocks $\mathbb {A}_1\mathbb {B}_1,\cdots, \mathbb{A}_r\mathbb {B}_r$.
\end{proof}

\section{The inverses of the triangular blocked tensors}

\hskip.6cm In \cite{2014BZZWW}, Bu et al. defined the left and right inverses of a tensor as following.

\begin{Def}\label{Def51}  Let $\mathbb{A}$ and $\mathbb{B}$ be tensors of dimension $n$ with order $m$ and $k$, respectively.
If $\mathbb{A}\mathbb{B}=\mathbb{I}$, then $\mathbb{A}$ is called a left $m$-inverse of $\mathbb{B}$, and $\mathbb{B}$ is called a
right $k$-inverse of $\mathbb{A}$.\end{Def}

In \cite{2014BZZWW}, Bu et al. obtained some results on left 2-inverses and right 2-inverses of tensors.

In \cite{2016LL}, W.Liu and W.Li further studied the left and right $k$-inverses of tensors.
They proved that a tensor $\mathbb{A}$ has a left (or right) $k$-inverse if and only if $\mathbb{A}$ has a left (or right) $2$-inverse.
They also proved the uniqueness of the left $k$-inverse, and obtained the expressions of the left and right $k$-inverses of $\mathbb{A}$.

In this section, we study the left and right $k$-inverses of an $(n_1,\cdots,n_r)$-UTB tensors of all the three types.
We will show that, if a first or second or third type $(n_1,\cdots,n_r)$-UTB tensor $\mathbb{A}$ has a left (or right) $k$-inverse,
then its unique left $k$-inverse is still an $(n_1,\cdots,n_r)$-UTB tensor of all the three types, and all of its right $k$-inverses are still
$(n_1,\cdots,n_r)$-UTB tensors of all the three types. Furthermore, by showing (in Theorem \ref{The54}) that the left $k$-inverse (if any)
of a weakly irreducible nonsingular $M$-tensor is a positive tensor, we show that the left $k$-inverse (if any) of a first or second or third
type normal $(n_1,\cdots,n_r)$-UTB nonsingular $M$-tensor is an $(n_1,\cdots,n_r)$-UTB tensor of all the three types all of whose diagonal blocks
are positive tensors. (The meaning of ``normal" can be found in Definitions \ref{Def61} and \ref{Def62} of \S 6.)

\subsection{The left inverses of the triangular blocked tensors}
\hskip.6cm The $i$-th row of a tensor $\mathbb{A}$ of order $m$ and dimension $n$, denoted by $R_i(\mathbb{A})$,
is the row-subtensor of order $m-1$ and dimension $n$ of $\mathbb{A}$ with the entries
$$(R_i(\mathbb{A}))_{i_2\cdots i_m}=a_{ii_2\cdots i_m}.$$

\begin{Def}\label{Def52} A tensor $\mathbb{A}$ of order $m$ and dimension $n$ is called row-subtensor diagonal, or simply row diagonal,
if all of its row-subtensors $R_1(\mathbb{A}),\cdots,R_n(\mathbb{A})$ are diagonal tensors. Namely, if we have
$$a_{ii_2\cdots i_m}=0\quad (\mbox {if} \ i_2,\cdots,i_m \ \mbox {are not all equal}).$$\end{Def}

\begin{Def}\label{Def53}{\rm(\cite{2010P})}  Let $\mathbb{A}$ be an order $m$ and dimension $n$ tensor.
Then the majorization  matrix $M(\mathbb{A})$ of $\mathbb{A}$ is the matrix of order $n$ with the entries
$$(M(\mathbb{A}))_{ij}=a_{ij\cdots j}\quad (\forall i,j=1,\cdots,n).$$\end{Def}

\begin{Pro}\label{Pro51}  Let $\mathbb{A}$ be an order $m$ and dimension $n$ tensor.
Then $\mathbb{A}$ is row diagonal if and only if there exists a matrix $P$ of order $n$ such that $\mathbb{A}=P\mathbb{I}_m$,
and in this case, $P=M(\mathbb{A})$.
\end{Pro}
\begin{proof} Sufficiency. If $\mathbb{A}=P\mathbb{I}_m$, then by the definition of tensor product we have
$$a_{ii_2\cdots i_m}=\sum_{i_1=1}^np_{ii_1}\delta_{i_1i_2\cdots i_m}=0\quad (\mbox {if} \ i_2,\cdots,i_m \ \mbox {are not all equal}),$$
and
\begin{equation}\label{eq51}
a_{ij\cdots j}=\sum_{i_1=1}^np_{ii_1}\delta_{i_1j\cdots j}=p_{ij}\quad (\forall i,j=1,\cdots,n). \end{equation}
Thus $\mathbb{A}$ is row diagonal, and $P=M(\mathbb{A})$.

Necessity. If $\mathbb{A}$ is row diagonal. Write $P=M(\mathbb{A})$ and $\mathbb{B}=P\mathbb{I}_m$.
Then by the sufficiency part we know that $\mathbb{B}$ is row diagonal with $M(\mathbb{B})=P=M(\mathbb{A})$.
Thus we have $\mathbb{A}=\mathbb{B}=P\mathbb{I}_m=M(\mathbb{A})\mathbb{I}_m$.
\end{proof}

The following result can be found in \cite{2016LL}  by W.Li et al. in a slightly different version.

\begin{Lem}\label{Lem51}{\rm(\cite{2016LL}, Theorem 3.1 and Corollary 3.3)} Let $k\ge 2$ and $\mathbb{A}$ be an order $m$ and dimension $n$ tensor.
Then the following three conditions are equivalent:

\noindent {\rm(1)} $\mathbb{A}=P\mathbb{I}_m$ for some nonsingular matrix $P$.

\noindent {\rm(2)} $\mathbb{A}$ has a left 2-inverse.

\noindent {\rm(3)} $\mathbb{A}$ has a left $k$-inverse.

\noindent And in this case, the unique left 2-inverse of $\mathbb{A}$ is $P^{-1}$,
and the unique left $k$-inverse of $\mathbb{A}$ is $\mathbb{I}_kP^{-1}$.
\end{Lem}

\noindent (Notice that by Proposition \ref{Pro51} we know that condition (1) is also equivalent to the following condition:

\noindent  {\rm(4)} $\mathbb{A}$ is row diagonal, and its majorization matrix $M(\mathbb{A})$ is nonsingular.)

\begin{proof}
See \cite{2016LL} Theorem 3.1 and Corollary 3.3.
\end{proof}

\begin{Lem}\label{Lem52}
Let $P$ be a matrix of order $n$, $\mathbb{A}=P\mathbb{I}_m$ be an order $m$ dimension $n$,  and a first or second or third type $(n_1,\cdots,n_r)$-UTB tensor. Then $P$ is an $(n_1,\cdots,n_r)$-UTB matrix.\end{Lem}

\begin{proof}
Since $\mathbb{A}$ is a first or second or third type $(n_1,\cdots,n_r)$-UTB tensor, we have $a_{it\cdots t}=0$ for any $j=2,\cdots,r$, $i\in I_j$ and $t\le S_{j-1}$. Now $\mathbb{A}=P\mathbb{I}_m$, so by (\ref{eq51}) we have $p_{it}=a_{it\cdots t}=0$ for any $j=2,\cdots,r$, $i\in I_j$ and $t\le S_{j-1}$, so by definition $P$ is an $(n_1,\cdots,n_r)$-UTB matrix.
\end{proof}


\begin{The}\label{The51}   Let $k\ge 2$, $\mathbb{A}$ be an order $m$ dimension $n$,  and  the first or second or third type $(n_1,\cdots,n_r)$-UTB tensor with diagonal blocks
$\mathbb {A}_1,\cdots, \mathbb{A}_r$. If $\mathbb{A}$ has a left $k$-inverse, then its (unique) left $k$-inverse is an $(n_1,\cdots,n_r)$-UTB tensor of all the three types, with the diagonal blocks $\mathbb {B}_1,\cdots, \mathbb{B}_r$, where $\mathbb{B}_i$ is the (unique) left $k$-inverse of
$\mathbb{A}_i \ (i=1,\cdots,r)$.\end{The}

\begin{proof} By Lemma \ref{Lem51} we see that $\mathbb{A}=P\mathbb{I}_m$ for some nonsingular matrix $P$, since $\mathbb{A}$ has a left $k$-inverse. By Lemma \ref{Lem52} we also know that $P$ is an $(n_1,\cdots,n_r)$-UTB matrix. Thus $P^{-1}$ is also an $(n_1,\cdots,n_r)$-UTB matrix. By Lemma \ref{Lem51} again we further know that its (unique) left $k$-inverse is $\mathbb{B}=\mathbb{I}_kP^{-1}$. Now both $\mathbb{I}_k$ and $P^{-1}$ are  $(n_1,\cdots,n_r)$-UTB tensors of all the three types, so by Theorem \ref{The41},
the product $\mathbb{B}=\mathbb{I}_kP^{-1}$ is  an $(n_1,\cdots,n_r)$-UTB tensor of all the three types.

Now $\mathbb{B}$ is the left $k$-inverse of $\mathbb{A}$. So by using Theorem \ref{The41} again,
$\mathbb{I}_{(k-1)(m-1)+1}=\mathbb{B}\mathbb{A}$ is an $(n_1,\cdots,n_r)$-UTB tensor, with the diagonal blocks
$\mathbb {B}_1\mathbb {A}_1,\cdots, \mathbb{B}_r\mathbb {A}_r$. But these blocks must all be identity tensors (as some diagonal blocks of $\mathbb{I}_{(k-1)(m-1)+1}$), so each $\mathbb{B}_i$ is the (unique) left $k$-inverse of $\mathbb{A}_i \ (i=1,\cdots,r)$.
\end{proof}

\subsection{The right inverses of the triangular blocked tensors}

\hskip.6cm Now we discuss the right inverses of the three types of UTB tensors.  The following result was obtained by Liu and Li in \cite{2016LL}.

\noindent \begin{Lem}\label{Lem53}{\rm(\cite{2016LL}, Theorem 3.4)}
Let $k\ge 2$ and $\mathbb{A}$ be an order $m$ dimension $n$ tensor. Then the following three conditions are equivalent:

\noindent {\rm(1)} There exists an nonsingular matrix $Q$ of order $n$, such that $\mathbb{A}=\mathbb{I}_mQ$.

\noindent {\rm(2)} $\mathbb{A}$ has a right 2-inverse.

\noindent {\rm(3)} $\mathbb{A}$ has a right $k$-inverse.\end{Lem}

 \begin{proof}  (2)$\Longleftrightarrow$(3) was already proved in \cite{2016LL} (Theorem 3.4).
For the convenience of the readers, here we use the results on left inverses to give a direct proof of (1)$\Longleftrightarrow$(3).

\noindent (1)$\Longrightarrow$(3): Obviously $Q^{-1}\mathbb{I}_k$ is a right $k$-inverse of $\mathbb{A}=\mathbb{I}_mQ$.

\noindent (3)$\Longrightarrow$(1): Let $\mathbb{B}$ be a right $k$-inverse of $\mathbb{A}$. Then $\mathbb{A}$ is the (unique) left $m$-inverse of
$\mathbb{B}$. So by Lemma \ref{Lem51} we have $\mathbb{B}=P\mathbb{I}_k$, and so $\mathbb{A}=\mathbb{I}_mP^{-1}$ for some nonsingular matrix $P$.
\end{proof}

The following Lemma \ref{Lem54}  will be used in the proof of Theorem \ref{The52}.

\begin{Lem}\label{Lem54} Let $Q$ be a matrix of order $n$, $\mathbb{A}=\mathbb{I}_mQ$ be an order $m$ dimension $n$,  and a first or second or third type $(n_1,\cdots,n_r)$-UTB tensor. Then $Q$ is an $(n_1,\cdots,n_r)$-UTB matrix. \end{Lem}

\begin{proof} Since $\mathbb{A}$ is a first or second or third type $(n_1,\cdots,n_r)$-UTB tensor, we have $a_{it\cdots t}=0$ for any $j=2,\cdots,r$, $i\in I_j$ and $t\le S_{j-1}$.

On the other hand, by the condition $\mathbb{A}=\mathbb{I}_mQ$ and the definition of the tensor product we also have:
$$a_{it\cdots t}=\sum_{j_2,\cdots ,j_m=1}^n \delta_{ij_2\cdots j_m}q_{j_2t}\cdots q_{j_mt} =q_{it}^{m-1}.$$
\noindent Thus we have $q_{it}^{m-1}=a_{it\cdots t}=0$ for all $j=2,\cdots,r$, $i\in I_j$ and $t\le S_{j-1}$. So by definition $Q$ is an $(n_1,\cdots,n_r)$-UTB matrix.\end{proof}

The following theorem and its proof hold for all three types of UTB tensors (since Lemma \ref{Lem54} holds for all three types of UTB tensors).

\begin{The}\label{The52}    Let $k\ge 2$, and $\mathbb{A}$ be an order $m$ dimension $n$, the first or second or third type $(n_1,\cdots,n_r)$-UTB tensor with diagonal blocks
$\mathbb {A}_1,\cdots, \mathbb{A}_r$, and $\mathbb{A}$ has a right $k$-inverse. Then any right $k$-inverse $\mathbb{B}$ of $\mathbb{A}$ is an $(n_1,\cdots,n_r)$-UTB tensor of all the three types, with the diagonal blocks $\mathbb {B}_1,\cdots, \mathbb{B}_r$,
where $\mathbb{B}_i$ is a right $k$-inverse of $\mathbb{A}_i \ (i=1,\cdots,r)$.
\end{The}\begin{proof}
Since $\mathbb{B}$ is a right $k$-inverse $\mathbb{A}$,  $\mathbb{B}$ has a left $m$-inverse $\mathbb{A}$. So by Lemma \ref{Lem51} we know that
$\mathbb{B}=P\mathbb{I}_k$ for some nonsingular matrix $P$, and in this case we have $\mathbb{A}=\mathbb{I}_mP^{-1}$.

By hypothesis $\mathbb{A}$ is a first or second or third type $(n_1,\cdots,n_r)$-UTB tensor, so by Lemma \ref{Lem54} we know that $P^{-1}$ (and thus $P$) is an $(n_1,\cdots,n_r)$-UTB matrix.

Now both $P$ and $\mathbb{I}_k$ are the $(n_1,\cdots,n_r)$-UTB tensors of all the three types, so by Theorem \ref{The41} we know that their product $\mathbb{B}$ is also an $(n_1,\cdots,n_r)$-UTB tensor of all the three types.

Also $\mathbb{B}$ is a right $k$-inverse of $\mathbb{A}$ with the diagonal blocks $\mathbb {B}_1,\cdots, \mathbb{B}_r$.
So by using Theorem \ref{The41} again,
$\mathbb{I}_{(k-1)(m-1)+1}=\mathbb{A}\mathbb{B}$ is an $(n_1,\cdots,n_r)$-UTB tensor, with the diagonal blocks
$\mathbb {A}_1\mathbb {B}_1,\cdots, \mathbb {A}_r\mathbb{B}_r$. But these blocks must all be identity tensors (as some diagonal blocks of $\mathbb{I}_{(k-1)(m-1)+1}$), so each $\mathbb{B}_i$ is a right $k$-inverse of $\mathbb{A}_i \ (i=1,\cdots,r)$.
\end{proof}

\subsection{The left $k$-inverses of the normal $(n_1,\cdots,n_r)$-UTB nonsingular $M$-tensors}

\hskip.6cm In \cite{2013DQW}, Ding, Qi and Wei generalized the concept of $M$-matrices to $M$-tensors.

A real tensor $\mathbb{A}$ is called a $Z$-tensor, if all of its off-diagonal entries are non-positive.
Equivalently, $\mathbb{A}$ is a $Z$-tensor if and only if $\mathbb{A}$ can be expressed as $\mathbb{A}=s\mathbb{I}-\mathbb{B}$,
where $\mathbb{B}$ is a nonnegative tensor, and $\mathbb{I}$ is the unit tensor.

A $Z$-tensor $\mathbb{A}=s\mathbb{I}-\mathbb{B}$ is called an $M$-tensor, if $s\ge\rho(\mathbb{B})$; And is called an nonsingular $M$-tensor,
if $s>\rho(\mathbb{B})$.

It is well-known in matrix theory (\cite{1994}) that if $A$ is an irreducible nonsingular $M$-matrix, then $A^{-1}$ is a positive matrix.

An $(n_1,\cdots,n_r)$-UTB tensor $\mathbb{A}$ of the first or second type is called the first or second type normal,
if all its diagonal blocks $\mathbb{A}_1,\cdots, \mathbb{A}_r$ are weakly irreducible. An $(n_1,\cdots,n_r)$-UTB tensor $\mathbb{A}$ of
the third type is called the third type normal, if all its diagonal blocks $\mathbb{A}_1,\cdots, \mathbb{A}_r$ are irreducible
(also see Definitions \ref{Def61} and \ref{Def62} in \S 6).

In this subsection, we will mainly prove the following Theorem 5.4.

\vskip.1cm
{\noindent {\bf Theorem 5.4} \it
If $\mathbb{A}$ is an order $m$ dimension $n$ and the first or second or third type normal $(n_1,\cdots,n_r)$-UTB
nonsingular $M$-tensor, and $\mathbb{A}$ has a left $k$-inverse $\mathbb{B}$. Then $\mathbb{B}$ is an  $(n_1,\cdots,n_r)$-UTB tensor of all the three types all of whose diagonal blocks $\mathbb{B}_1,\cdots, \mathbb{B}_r$ are positive tensors.}

In order to prove Theorem 5.4, we first need to prove some preliminary results (The following Lemma \ref{Lem55} and Theorem \ref{The53}).

\begin{Lem}\label{Lem55}   If an order $m$ dimension $n$ row diagonal tensor $\mathbb{A}=P\mathbb{I}_m$, where $P$ is a matrix of order $n$.
Then we have:

\vskip 0.1cm

\noindent {\rm(1)} $\det(\lambda\mathbb{I}_m-\mathbb{A})=(\det(\lambda\mathbb{I}_2-P))^{(m-1)^{n-1}}$.

\vskip 0.1cm

\noindent {\rm(2)} $\mathbb{A}$ is a $Z$-tensor if and only if $P$ is a $Z$-matrix.

\vskip 0.1cm

\noindent {\rm(3)} $\mathbb{A}$ is weakly irreducible if and only if $P$ is irreducible.

\vskip 0.1cm

\noindent {\rm(4)} $\mathbb{A}$ is an $M$-tensor if and only if $P$ is an $M$-matrix.

\vskip 0.1cm

\noindent {\rm(5)} $\mathbb{A}$ is a nonsingular $M$-tensor if and only if $P$ is a nonsingular $M$-matrix.
\end{Lem}

\begin{proof} (1). We have
$$\lambda\mathbb{I}_m-\mathbb{A}=\lambda\mathbb{I}_2\mathbb{I}_m-P\mathbb{I}_m=(\lambda\mathbb{I}_2-P)\mathbb{I}_m.$$
Thus by the formula of the determinants of the product of tensors (\cite{2013SSZ}) we have
$$\det(\lambda\mathbb{I}_m-\mathbb{A})=(\det(\lambda\mathbb{I}_2-P))^{(m-1)^{n-1}}(\det\mathbb{I}_m)
=(\det(\lambda\mathbb{I}_2-P))^{(m-1)^{n-1}}.$$

\noindent (2). Since $\mathbb{A}$ is row diagonal, we have $a_{ii_2\cdots i_m}=0$ if $i_2,\cdots ,i_m$ are not all equal. Thus we have:
$$\mathbb{A} \ \mbox {is a $Z$-tensor}\Longleftrightarrow a_{ij\cdots j}\le 0 \ (\forall i\ne j)
\Longleftrightarrow p_{ij}\le 0 \ (\forall i\ne j)\Longleftrightarrow P \  \mbox { is a $Z$-matrix}.$$

\vskip 0.1cm

\noindent (3). By Proposition \ref{Pro51} we know that $P=M(\mathbb{A})$ is the majorization matrix of $\mathbb{A}$.

\vskip 0.1cm

On the other hand, let $G(|\mathbb{A}|)$ be the ``representation matrix" of $|\mathbb{A}|$ defined as (\cite{2011YY}):
\begin{equation}\label{eq51}(G(|\mathbb{A}|))_{ij}=\sum_{\{i_2,\cdots ,i_m\}\ni j}|a_{ii_2\cdots i_m}|.\end{equation}
Then by the hypothesis that $\mathbb{A}$ is row diagonal we can verify that
\begin{equation}\label{eq51}(G(|\mathbb{A}|))_{ij}=\sum_{\{i_2,\cdots ,i_m\}\ni j}|a_{ii_2\cdots i_m}|=|a_{ij\cdots j}|
=(|M(\mathbb{A})|)_{ij}.\end{equation}
Thus we have $G(|\mathbb{A}|)=|M(\mathbb{A})|$. So by \cite{2011YY} (Definitions 2.1 and 2.5) we have
$$\mathbb{A} \ \mbox {is weakly irreducible}\Longleftrightarrow |\mathbb{A}| \ \mbox {is weakly irreducible}
\Longleftrightarrow G(|\mathbb{A}|) \ \mbox {is irreducible}$$

\hskip4.8cm $\Longleftrightarrow |M(\mathbb{A})|$ is irreducible $\Longleftrightarrow M(\mathbb{A})$ is irreducible

\hskip4.8cm $\Longleftrightarrow P$ is irreducible.

\vskip 0.1cm

\noindent (4). By (1) we have that
$$\mbox {The real parts of all eigenvalues of $\mathbb{A}$ are} \ \ge 0
\Longleftrightarrow \mbox{The real parts of all eigenvalues of $P$ are} \ \ge 0.$$
\noindent So we have
$$\begin{aligned}
&\mathbb{A} \ \mbox {is an $M$-tensor}\\
\Longleftrightarrow &\mathbb{A} \ \mbox {is a $Z$-tensor and the real parts of all eigenvalues of $\mathbb{A}$ are} \ \ge 0\\
\Longleftrightarrow &P \ \mbox {is a $Z$-matrix and the real parts of all eigenvalues of $P$ are} \ \ge 0\\
\Longleftrightarrow &P \ \mbox {is an $M$-matrix}.
\end{aligned}$$

\vskip 0.1cm

\noindent (5). Replace all ``$\ge 0$" by ``$>0$", and ``$M$-tensor" by ``nonsingular $M$-tensor" in the proof of (4).
\end{proof}

From Lemma \ref{Lem55} we can obtain the following result which is a generalization of the corresponding result for irreducible nonsingular $M$-matrices.

\begin{The}\label{The53}   Let $\mathbb{A}$ be an order $m$ dimension $n$ weakly irreducible nonsingular $M$-tensor,
and $\mathbb{A}$ has a left $k$-inverse $\mathbb{B}$. Then $\mathbb{B}>0$.\end{The}
\begin{proof} Since $\mathbb{A}$ has a left $k$-inverse, we have that $\mathbb{A}=P\mathbb{I}_m$ (for some matrix $P$) is row diagonal by Lemma \ref{Lem51}.
Thus by (3) and (5) of Lemma \ref{Lem55}, $P$ is an irreducible nonsingular $M$-matrix. By the result of $M$-matrices (\cite{1994}),
we have $P^{-1}>0$.

On the other hand, we know that the left $k$-inverse $\mathbb{B}=\mathbb{I}_kP^{-1}$. Write $Q=P^{-1}$, then by the definition of the tensor product
we have
$$b_{ii_2\cdots i_k}=\sum_{j_2,\cdots ,j_k=1}^n \delta_{ij_2\cdots j_k}q_{j_2i_2}\cdots q_{j_ki_k} =q_{ii_2}\cdots q_{ii_k}>0.$$
Thus we have $\mathbb{B}=\mathbb{I}_kP^{-1}>0$.
\end{proof}

Now we are ready to prove Theorem 5.4.

{\bf\noindent Proof of Theorem 5.4:} Let the diagonal blocks of $\mathbb{A}$ be $\mathbb{A}_1,\cdots, \mathbb{A}_r$
(all of which are weakly irreducible by hypothesis).

Firstly, by Theorem \ref{The51} we know that $\mathbb{A}$ has a left $k$-inverse implying that each $\mathbb{A}_i$ also has a left $k$-inverse.

Secondly, we show that $\mathbb{A}$ is a nonsingular $M$-tensor implying that each $\mathbb{A}_i$ is also a nonsingular $M$-tensor:

Write $\mathbb{A}=s\mathbb{I}_m-\mathbb{C}$, where $\mathbb{C}\ge 0$
and $s>\rho (\mathbb{C})$. Then $\mathbb{C}=s\mathbb{I}_m-\mathbb{A}$ is an $(n_1,\cdots,n_r)$-UTB tensor with the diagonal blocks,
say $\mathbb{C}_1,\cdots, \mathbb{C}_r$. Thus we have
$$\mathbb{A}_i=s\mathbb{I}_m(n_i)-\mathbb{C}_i \qquad (i=1,\cdots,r),$$
where $\mathbb{I}_m(n_i)$ is the unit tensor of order $m$ and dimension $n_i$.

\vskip 0.1cm

Now $s>\rho(\mathbb{C})=\max_{1\le i\le r}\rho(\mathbb{C}_i)$ implying that each $\mathbb{A}_i$ is also a nonsingular $M$-tensor.

Thus we see that this weakly irreducible $\mathbb{A}_i$ satisfies all the conditions of Theorem \ref{The53}. So by Theorem \ref{The53}
we conclude that its left $k$-inverse $\mathbb{B}_i>0 \ (i=1,\cdots,r)$.

\section{The reducible and weakly reducible normal form of tensors}

\hskip.6cm In this section, we study the reducible and weakly reducible normal forms of tensors.
Since all the problems and results considered in this section depend only on the zero-nonzero pattern of the tensors,
we may assume without loss of generality that all the tensors considered in this section are nonnegative tensors.

First we recall some definitions and results in the matrix case.

An $(n_1,\cdots,n_r)$-UTB matrix $A$ is called normal (or canonical), if all of whose diagonal blocks are irreducible.
This concept can be generalized to tensors as follows.

\begin{Def}\label{Def61}  An $(n_1,\cdots,n_r)$-UTB (or 2ndUTB) tensor $\mathbb{A}$ is called a first type (or second type)
(weakly reducible) normal tensor, if all of whose diagonal blocks are weakly irreducible. \end{Def}

\begin{Def}\label{Def62} An $(n_1,\cdots,n_r)$-3rdUTB tensor $\mathbb{A}$ is called a third type (reducible) normal tensor,
if all of whose diagonal blocks are irreducible.\end{Def}

In the matrix case, we have the following two well-known results about the existence and uniqueness of the normal form of a matrix of order $n$.

\begin{The}\label{The61}{\rm(\cite{1991})} Every matrix $A$ of order $n$ is permutational similar to some normal upper (or lower)
triangular blocked matrix.\end{The}

\begin{The}\label{The62}{\rm(\cite{1991})} Let $A$ be an $(n_1,\cdots,n_r)$ normal upper triangular blocked matrix of order $n$ with
diagonal blocks $A_1,\cdots,A_r$, and $B$ be an $(m_1,\cdots,m_t)$ normal upper triangular blocked matrix of order $n$ with diagonal
blocks $B_1,\cdots,B_t$, respectively. If $A$ and $B$ are permutational similar, then $r=t$, and there exist some permutation $\sigma$:
$[r]\rightarrow [r]$, such that $A_i$ and $B_{\sigma(i)}$ are permutational similar for $i=1,\cdots,r$.\end{The}

Shao et al in \cite{2013SSZ} and Hu et al in \cite{2014HHQ} obtained that every order $m$ dimension $n$ tensor is permutation similar to some normal 2ndUTB tensor (also see Proposition 1 of \cite{2016HQ}), thus generalized Theorem \ref{The61} from matrices to the second type normal upper triangular blocked tensors. In this section, we will show in Theorem \ref{The63} that Theorem \ref{The61} can also be generalized to the third type normal upper triangular blocked tensors. We also give an example to show that not every tensor can be permutational similar to some first type normal upper triangular blocked tensor.

\vskip 0.1cm

First we have the following lemma.

\begin{Lem}\label{Lem61} {\rm(1)} If an order $m$ dimension $n$ tensor $\mathbb{A}$ is $I$-weakly reducible with $|I|=n-k \ (1\le k\le n-1)$.
Then $\mathbb{A}$ is permutation similar to some $(k,n-k)$-UTB (and thus 2ndUTB by Remark \ref{Rem21}) tensor.

\noindent {\rm(2)} If $\mathbb{A}$ is $I$-reducible with $|I|=n-k \ (1\le k\le n-1)$. Then $\mathbb{A}$ is permutation similar to some
$(k,n-k)$-3rdUTB  tensor.\end{Lem}

\begin{proof} Take the permutation $\sigma: [n]\rightarrow [n]$, such that $\sigma(I)=\{k+1,\cdots,n\}$.
Take the permutation matrix $P=P_{\sigma}$. Then the tensor $P\mathbb{A}P^T$ is a $(k,n-k)$-UTB (or 3rdUTB) tensor which
is permutation similar to $\mathbb{A}$.\end{proof}

\begin{The}\label{The63}  Every order $m$ dimension $n$ tensor $\mathbb{A}$ is permutation similar to some normal 3rdUTB tensor.\end{The}

\begin{proof} We use induction on $n$. If $\mathbb{A}$ is irreducible, then $\mathbb{A}$ itself is already in the required form.
If $\mathbb{A}$ is reducible, then by Lemma \ref{Lem61} there exists some $k$ with $1\le k\le n-1$ such that $\mathbb{A}$ is permutation
similar to some $(k,n-k)$-3rdUTB tensor $\mathbb{B}$ with two diagonal blocks $\mathbb{B}_1$ and $\mathbb{B}_2$.

\noindent (The following blocked forms for the matrix cases illustrate the ideas of the proof.)

$$\mathbb{A}\sim \mathbb{B}=\left (
\begin{array}{cc}
B_{1}&\ast\\
O&B_2\\
\end{array} \right )
$$
By induction we know that $\mathbb{B}_1$ and $\mathbb{B}_2$
are respectively permutation similar to some normal 3rdUTB tensors $\mathbb{C}$ and $\mathbb{D}$ all of whose diagonal blocks $\mathbb{C}_1,\cdots,\mathbb{C}_p$ and $\mathbb{D}_1,\cdots,\mathbb{D}_q$ are irreducible.

$$\mathbb{B}_1\sim \mathbb{C}=\left (
\begin{array}{ccc}
\mathbb{C}_{1}&\ldots & \ast\\
\vdots &\ddots &\vdots\\
O&\ldots & \mathbb{C}_p\\
\end{array} \right ), \  \   \mathbb{B}_2\sim \mathbb{D}=\left (
\begin{array}{ccc}
\mathbb{D}_{1}&\ldots & \ast\\
\vdots &\ddots &\vdots\\
O&\ldots & \mathbb{D}_q\\
\end{array} \right ) \Longrightarrow \mathbb{B}\sim \left (
\begin{array}{cc}
\mathbb{C}&\ast\\
O&\mathbb{D}\\
\end{array} \right ) := \mathbb{A}^*,
$$
where $\mathbb{A}^*$ defined above is a $(k,n-k)$-3rdUTB tensor with two diagonal blocks $\mathbb{C}$ and $\mathbb{D}$.
By Theorem \ref{The21} we know that the tensor $\mathbb{A}^*$ is a 3rdUTB tensor with the diagonal blocks $\mathbb{C}_1,\cdots,\mathbb{C}_p$ and $\mathbb{D}_1,\cdots,\mathbb{D}_q$ (they are all irreducible), so $\mathbb{A}^*$ is a normal 3rdUTB tensor. But we also have $\mathbb{A}\sim \mathbb{B}\sim \mathbb{A}^*$, thus we obtained the desired result.
\end{proof}

Using the proof similar to that of Theorem \ref{The63}, we can also prove the following result for the second type weakly reducible normal form which was also obtained by Shao et al in \cite{2013SSZ} and Hu et al in \cite{2014HHQ}.

\begin{The}\label{The64}  Every order $m$ dimension $n$ tensor $\mathbb{A}$ is permutation similar to some normal 2ndUTB tensor.\end{The}

\vskip 0.1cm
The following example shows that not every tensor can be permutational similar to some first type normal upper triangular blocked tensor.

\begin{Exa}\label{Exa61} Let $\mathbb{A}$ be a $(0,1)$ tensor of order $m=3$ and dimension $n=4$ with the entries
$$a_{ijk}=\left\{
                            \begin{array}{cc}
                              1, &   \mbox{  if  } 1\le i\le 3, \ 4\in \{j,k\};  \\
                              0, &  \mbox{otherwise.}\\
                            \end{array}
                          \right.$$
Then $\mathbb{A}$ is not permutational similar to any first type normal upper triangular blocked tensor.\end{Exa}
\begin{proof} Firstly, $\mathbb{A}$ is weakly reducible since the fourth row of $\mathbb{A}$ is a zero row. Now if $\mathbb{A}$ is
$\sigma$-permutational similar to some first type normal $(n_1,\cdots,n_r)$-upper triangular blocked tensor $\mathbb{B}$ with diagonal blocks
$\mathbb{B}_1,\cdots,\mathbb{B}_r$, then $r\ge 2$. By Lemma \ref{Lem21} we know that $\mathbb{B}$ satisfies the following two conditions
(write $I_1=[n_1]$):

\vskip 0.1cm

\noindent {\rm(i)} $\mathbb{B}[I_1]=\mathbb{B}_1$ is weakly irreducible.

\vskip 0.1cm

\noindent {\rm(ii)} $\mathbb{B}$ is $\overline {I_1}$-weakly reducible.

Now take $I=\sigma (\overline {I_1})$, then we see that $I$ is a proper subset of $[n]$ satisfying the following two conditions:

\vskip 0.1cm

\noindent {\rm(1)} $\mathbb{A}[\overline {I}]$ is weakly irreducible.

\vskip 0.1cm

\noindent {\rm(2)} $\mathbb{A}$ is $I$-weakly reducible.

\vskip 0.1cm

Now we consider the following two cases.

\vskip 0.18cm

\noindent {\bf Case 1:} $4\notin I$.

Take any $i\in I$. Then $i\ne 4$, so by definition we have $a_{i44}=1$, contradicting (2) that $\mathbb{A}$ is $I$-weakly reducible.

\noindent {\bf Case 2:} $4\in I$.

\vskip 0.18cm

\noindent {\bf Subcase 2.1:} $|\overline {I}|\ge 2$.

Then $4\in I\Longrightarrow \overline {I}\subseteq \{1,2,3\}$. Thus by definition we have $\mathbb{A}[\overline {I}]=0$, contradicting (1) that
$\mathbb{A}[\overline {I}]$ is weakly irreducible (since $|\overline {I}|\ge 2$).

\noindent {\bf Subcase 2.2:} $|\overline {I}|=1$.

Without loss of generality we may assume that $\overline {I}=\{1\}$. Then by definition we have $a_{214}=1$, contradicting (2)
that $\mathbb{A}$ is $I$-weakly reducible.
\end{proof}

\begin{Rem}\label{Rem61}
Notice that we have not generalized the uniqueness of the normal form for matrices
in Theorem \ref{The62} to all the three types normal UTB tensors. We think that these possible generalizations could be the problems
for further study.\end{Rem}

\end{document}